\documentclass[12pt,a4paper]{amsart}
\usepackage{amssymb, amstext, amscd, amsmath, color}

\usepackage{url}

\usepackage[dvips]{graphicx}

\usepackage{amsfonts}
\usepackage{amsthm}
\usepackage{amsfonts}
\usepackage{array}
\usepackage{epsfig}
\usepackage{eucal}
\usepackage{latexsym}
\usepackage{mathrsfs}
\usepackage{textcomp}
\usepackage{verbatim}

\newtheorem{thm}{Theorem}[section]

\newtheorem{cor}[thm]{Corollary}

\newtheorem{defn}[thm]{Definition}

\newtheorem{lem}[thm]{Lemma}

\newtheorem{prop}[thm]{Proposition}
\newtheorem{rem}[thm]{Remark}

\numberwithin{equation}{section}

\newcommand{\bQ}{{\mathbb{Q}}}
\newcommand{\bR}{{\mathbb{R}}}

\newcommand{\bZ}{{\mathbb{Z}}}

\newcommand{\bF}{{\mathbb{F}}}

  \newcommand{\C}{{\mathcal{C}}}

  \newcommand{\M}{{\mathcal{M}}}
  \newcommand{\N}{{\mathcal{N}}}

  \newcommand{\R}{{\mathcal{R}}}

  \newcommand{\U}{{\mathcal{U}}}

  \newcommand{\X}{{\mathcal{X}}}

\newcommand{\rank}{\operatorname{rank}}

\begin{document}

\title[Rigidity of frameworks supported on surfaces]{Rigidity of frameworks supported on surfaces}
\author[A. Nixon]{A. Nixon}
\address{Dept.\ Math.\ Stats.\\ Lancaster University\\
Lancaster LA1 4YF \\U.K. }
\email{a.nixon1@lancaster.ac.uk}
\thanks{First author supported by  EPSRC grant EP/P503825/1}
\author[J.C. Owen]{J.C. Owen}
\address{D-Cubed, Siemens PLM Software, Park House,\\
Castle Park, Cambridge UK}
\email{owen.john.ext@siemens.com}
\author[S.C. Power]{S.C.  Power}
\address{Dept.\ Math.\ Stats.\\ Lancaster University\\
Lancaster LA1 4YF \\U.K. }
\email{s.power@lancaster.ac.uk}


\thanks{2000 {\it  Mathematics Subject Classification.}
52C25, 05B35, 05C10 \\
Key words and phrases:
bar-joint framework, framework on a surface, rigid framework}

\begin{abstract}
A theorem of Laman gives a combinatorial characterisation of the graphs that
admit a realisation as a minimally rigid generic bar-joint framework in $\bR^2$.
A more general theory is developed
for frameworks in $\bR^3$ whose vertices are constrained to move on
a two-dimensional smooth submanifold $\M$. Furthermore,
when $\M$ is a union of concentric spheres, or a
union of parallel planes or a union of concentric
cylinders, necessary and sufficient
combinatorial conditions are obtained for the minimal rigidity of generic
frameworks.
\end{abstract}

\date{}
\maketitle


\section{Introduction}
\label{paper1sec1}

A bar-joint framework realisation of a simple finite connected graph
$G=(V,E)$ is a pair $(G,p)$ where
$p=(p_1,\dots ,p_n)$ is
an assignment of the vertices $v_1, \dots ,v_n$ in $V$ to framework points in $\bR^d$. In the case of frameworks in the plane, there is a celebrated characterisation
of those graphs $G$ whose typical frameworks  are both rigid and minimally rigid.
By \textit{rigid} we mean that any edge-length-preserving motion
 is necessarily a rigid motion. That is, a continuous edge-length-preserving path
$p(t), t \in [0,1]$, with $p(0)=p$,  is necessarily induced by a continuous path
of isometries of $\bR^d$.
 The function $p(t)$ is known as a \textit{continuous flex} of the framework
$(G,p)$ and \textit{minimal rigidity} means that the framework is rigid with the
removal of any framework edge resulting in a nonrigid framework.

In the following theorem{, due to Laman \cite{Lam}} the term
\textit{generic}  means that the framework coordinates of $(G,p)$, of which there are
$2|V|$ in number, are algebraically independent over $\bQ$.
This is one way of formalising the notion of a ``typical'' framework for $G$.

\begin{thm} A finite connected simple graph $G=(V,E)$ {with $|V|\geq 2$} admits a minimally
rigid generic realisation $(G,p)$ in $\bR^2$ if and only if
\begin{enumerate}
\item[(i)] $2|V|-|E| = 3$ and
\item[(ii)] $2|V'|-|E'|  \geq 3$ for every subgraph $G'=(V',E')$ with $|E'|>1$.
\end{enumerate}
Moreover every generic realisation $(G,p)$ of such a graph is minimally rigid.
\end{thm}

There is a well-known notion of \textit{infinitesimal
rigidity},
which coincides with rigidity in the case of generic frameworks. See Gluck \cite{Glu}
and Asimow and Roth \cite{A&R} for example. However
frameworks may be infinitesimally
flexible while also being (continuously) rigid so this is a stronger notion.
The theorem above is due to Laman \cite{Lam} in its infinitesimally rigid formulation.

A graph  satisfying  (ii) above is said to be an \textit{independent graph} for the plane,
or simply an independent graph when the context is understood.
The terminology here relates to
the connections between the rigidity of geometric framework structures and
the theory of matroids. We shall not need these connections which may be found, for example,   in Graver, Servatius and Servatius
 \cite{GSS}, Jackson and Jordan \cite{J&J}, \cite{Jac} and Whiteley \cite{Whi5}. When both (i) and (ii) hold then $G$ is said to be a {maximally independent graph} for the plane.
 { These graphs are also referred to} as {Laman graphs} or \textit{$(2,3)$-tight graphs}.

In what follows we analyse frameworks
$(G,p)$ supported on general smooth surfaces $\M$ embedded in $\bR^3$.
In particular  in Section \ref{sec:fonsurface},
we define continuous and infinitesimal rigidity and show that these
notions are equivalent for completely regular frameworks in the sense of Definition \ref{d:seqreg}. Also
we define the ambient degrees of freedom of a framework on a surface $\M$ and obtain necessary counting conditions for minimally rigid completely regular realisations.
The development here is in  the spirit of the
well-known characterisations of rigidity for free frameworks
given  by Asimow and Roth \cite{A&R}, \cite{Rot}, \cite{A&R2}, where regular frameworks were identified as the appropriate topologically generic notion.
The primary construct in rigidity theory is the rigidity matrix and for a framework $(G,p)$ on $\M$ we form a relative rigidity matrix $R_{\M}(G,p)$, with
$|E|+|V|$ rows and $3|V|$ columns, which incorporates the local
 { normal} vectors for $\M$ at the framework points. While we restrict
attention to embedded surfaces in $\bR^3$ there are
 {natural variants of these connections} in higher dimensions, as is also the case in Asimow and Roth \cite{A&R}.

In Section \ref{sec:henmconf}, we pay particular attention to the construction of  Henneberg moves between frameworks (rather than graphs) which
preserve minimal rigidity. These constructions together with the
graph theory of Section \ref{sec:grapht} are the {central} ingredients in the proof of the main result, Theorem \ref{t:lamancylinder}. This shows that there is a precise version of Laman's theorem for frameworks on a  circular cylinder with the class of $(2,2)$-tight graphs (see Definition \ref{d:indptgraphs}) playing the appropriate role.

The approach below embraces reducible surfaces and varieties and
we { also} obtain variants of Laman's theorem for frameworks supported on parallel planes, on concentric spheres and on concentric cylinders. As a direct corollary of this for the spheres and planes cases we recover  some results of Whiteley \cite{Whi2}  on the rigidity of cone frameworks in $\bR^3$. On the other hand  from the { concentric} cylinders case we deduce a novel variant for point-line frameworks in $\bR^3$
 { which have} a single line.

{We begin with some pure graph theory for $(2,3)$-tight and $(2,2)$-tight graphs
and show that, with the exception of the singleton graph $K_1$ each $(2,2)$-tight graph is generated from $K_4$ by the usual Henneberg moves together with the new extension move, as given in Definition \ref{d:extensionmove}. The graph extension move is similar in spirit to the $2$-sum move used by Berg and Jordan \cite{B&J} (along with the Henneberg $2$ move) to generate all circuits for the generic rigidity matroid in two dimensions.
However the $2$-sum move, combined with the Henneberg moves, is not sufficient for our purposes. This is revealed by the $(2,2)$-tight graph formed by two copies of $K_4$ sharing a single vertex.}

{
There are spanning tree characterisations of $(2,3)$-tight and $(2,2)$-tight multi-graphs which derive from a celebrated combinatorial result of Nash-Williams  \cite{N-W}, \cite{Tut}, \cite{L&Y}, and such equivalences have proven useful in rigidity theory for the locally flat
contexts, where multi-graphs play a role. Although we do not need the spanning tree viewpoint we nevertheless derive a spanning tree characterisation for $(2,2)$-tight graphs which are simple, Theorem \ref{t:nash-williams2}.
}

In all cases we are concerned with the usual Euclidean distance in $\bR^3$ rather than surface geodesics or other distance measures.
We note that Whiteley \cite{Whi5} and  Saliola and Whiteley \cite{S&W}  examine first order rigidity for spherical spaces
and various spaces where there is local flatness.  (See also Connelly and Whiteley
\cite{C&W} for global rigidity concerns and Schulze and Whiteley \cite{S&W} for transfer of metric results.) For the sphere
there is an equivalence between the direct distance and geodesic distance viewpoints which may be exploited. However this is a {very} special case and in general one must take account of curvature and local geometry.
Thus on the flat cylinder, derived
from $\bR^2/\bZ$ and direct distance in $\bR^2$,
a generic $K_4$ framework with no
wrap-around edges has three (independent) infinitesimal motions, while  a typical $K_4$
framework on the classical curved cylinder has only two.

The topic of frameworks constrained to surfaces is developed further in the sequel
\cite{NOP2} where a combinatorial characterisation has been obtained for minimally rigid
bar-joint framework on surfaces invariant under a \textit{singly} generated isometry group. These so-called type $1$ surfaces include the standard cone and torus, as well as surfaces of revolution and helicoids.

\section{Graph Theory}
\label{sec:grapht}

The \textit{Henneberg 2 move} (see e.g \cite{GSS})
 is an operation $G \to G'$ on
simple connected graphs
in which a new vertex of degree 3 is introduced by breaking an edge $v_iv_j$
into two edges $v_iw, v_jw$ at a
new vertex $w$ and adding an edge $wv_k$ to some other vertex $v_k$ of $G$.
The operation maps the set of  independent graphs (for the plane) to itself and also
preserves maximal independence. A key step in the standard proof of Laman's theorem
is to show that if the independent graph $G$ has a minimally rigid generic framework
realisation then so too does $G'$, and in Section \ref{sec:henmconf} we pursue this
in wider generality for Henneberg moves on  \textit{frameworks}
on smooth surfaces. However, for such rigidity preservation arguments to be sufficient to characterise minimal rigidity we also need to know that the
desired class of graphs can be derived inductively by such tractable moves or related moves. This is a purely
graph theoretical issue  and we now address this for Laman graphs (see Definition \ref{d:indptgraphs}(a)), Laman-plus-one
graphs (Definition \ref{def:plusone}) and $(2,2)$-tight graphs (Definition \ref{d:indptgraphs}(b)).

A \textit{Henneberg  1  move} (see e.g. \cite{GSS}) or vertex
addition move $G \to G'$ is the process of adding a degree two vertex with
two new  edges which are incident to any two distinct points of $G$.

\begin{prop}\label{p:lamangraph} Every Laman graph $G$ arises from a sequence
\[
G_0 \to G_1 \to \dots \to G_n=G
\]
where $G_0 = K_2$, the complete graph on two vertices and
where $G_k\to G_{k+1}$ is either a Henneberg $1$ move or a Henneberg $2$ move.
\end{prop}

The starting point for the proof of this fact is the observation that if $G$ is Laman
with no degree $2$ vertex then
there are at least $6$ vertices of degree $3$.
Indeed, if $n_i$ is the degree of the $i^{th}$ vertex then
$\Sigma_i n_i = 2|E|$ and so
\[
6=4|V|-2|E|= \sum_i(4-n_i).
\]
On any of these vertices
there is a way of performing an inverse Henneberg $2$ move on $G$ to create a Laman
graph. This was established by Laman \cite{Lam}
and requires some care  for one can
easily see that there are non-Laman graphs which become Laman after
a particular Henneberg 2 move.

Define $f(H)=2|V(H)|-|E(H)|$
for any graph $H=(V(H),E(H))$. This could be referred to as the \textit{freedom number} of $H$
(representing a sense of the total degrees of freedom when the vertices are viewed as having two degrees of freedom).
We remark that the definition of
a graph $(V,E)$ entails that  $|V|\geq 1$ and $|E|\geq 0$.

\begin{defn}\label{d:indptgraphs}
(a) A  graph
$G$ is \emph{$(2,3)$-sparse} if
$f(H) \geq 3$ for all subgraphs $H$
containing at least one edge
and is \emph{$(2,3)$-tight} if it is $(2,3)$-sparse and $f(G) = 3$.

(b) A graph
$G$ is \emph{$(2,2)$-sparse} if
$f(H) \geq 2$ for all subgraphs $H$
and is \emph{$(2,2)$-tight} if it is $(2,2)$-sparse and $f(G) = 2$.
\end{defn}

Recall that $k$-connectedness
means that if fewer than $k$ vertices are removed from a  graph then it remains connected.
One can readily check that while a Laman graph is $2$-connected, a $(2,2)$-tight graph is in general just $1$-connected.

{
The next elementary lemma is useful in the construction of tight graphs and also plays a role in the proof of  Lemma \ref{l:constraintgraph1}}

\bigskip
\begin{lem}\label{l:union} Let $r = 2$ or $3$.
Let $G$ be $(2,r)$-sparse with subgraphs
 $G_1$ and $G_2$  which
are $(2,r)$-tight. If $f(G_1 \cap G_2) \geq r$ then 
$G_1 \cap G_2$ and $G_1 \cup G_2$ are $(2,r)$-tight.
\end{lem}

\begin{proof}
As a subgraph of $G$, $f(G_1 \cup G_2) \geq r$.
We have
\[
f(G_1 \cup G_2)+f(G_1 \cap G_2) = f(G_1) + f(G_2)=2r
\]
and so 
\[f(G_1 \cup G_2) =f(G_1 \cap G_2)=r.
\]
\end{proof}

{
The next lemma (see for example \cite{GSS}, \cite{Lam} and \cite{whi-1})
provides the key for proof of Proposition \ref{p:lamangraph}.
Its analogue for degree $2$ vertices is elementary.}

\begin{lem}\label{l:constraintgraph1}
Let  $G$ be a $(2,3)$-tight graph with a degree $3$ vertex. Then there is a $(2,3)$-tight graph $G'$ with a Henneberg $2$ move $G' \to G$.
\end{lem}

We now discuss a particular class of $(2,2)$-tight graphs.

\begin{defn}\label{def:plusone}
A graph $G=(V,E)$ is a \textit{Laman-plus-one graph} if it is connected and simple,
with no degree $1$ vertices and is  such that { for some edge $e$} the graph
 $G\setminus e=(V, E\setminus e)$ is a Laman graph.
 \end{defn}

Note that if $G$ is constructed as two copies of $K_4$ joined at a common vertex, or joined by two connecting edges, then $G$ is $(2,2)$-tight but is not a Laman-plus-one graph.

{ The next proposition is due to Haas et al \cite{haa}. It may be proven by first noting that for a Laman graph one has the
vertex degree counting equation $6= \sum_i(4-n_i)$, where $n_i$ is the degree of the
$i$th vertex. Accordingly if there are no vertices of degree $2$ then there are a number of vertices of degree $3$. By examining the various cases it can be shown that the addition of an edge cannot inhibit  all the potential inverse Henneberg moves on the remaining  vertices of degree $3$, except in the case that $G$ is $K_4$.}

\begin{prop}\label{p:lamanplusonegraph}
Every Laman-plus-one graph is obtained from $K_4$ by a sequence of Henneberg $1$
and $2$ moves.
\end{prop}

{
\begin{rem}{\rm
We remark that the subclass of $3$-connected Laman graphs is relevant to
the Galois nonsolvability of frameworks in the plane and here one needs alternative moves for an inductive analysis. See Owen  and Power \cite{O&P}. (The general problem
in this area remains open.) Such alternative moves include vertex splitting, a move which  also features in the derivability
of $(2,1)$-tight graphs \cite{N&O} but which is not needed in the $(2,2)$-tight case.}
\end{rem}
}

\begin{rem}{\rm A simple connected graph is said to be a \textit{generically rigid graph} for the plane if it is rigid as a framework in $\bR^2$  in some vertex-generic realisation. In view of Laman's theorem this means that $G$ is a Laman graph plus some number of extra edges. More strongly, a graph $G$ is \textit{redundantly rigid} if it is rigid and remains so on removal of any edge. Redundant rigidity is plainly stronger than being Laman-plus-one and is intimately tied up with the topic of global (unique realisation) rigidity. We remark that the globally rigid graphs in the plane are $K_2$, $K_3$ and those that are derivable from $K_4$ by Henneberg 2 moves plus edge additions. This rather deeper result is discussed in Jackson and Jordan   \cite{J&J}, \cite{Jac}.}
\end{rem}

The following  lemma is the key
for bridging the gap between Laman-plus-one graphs and $(2,2)$-tight graphs.

\begin{lem}\label{l:alternative}
Let $G$ be a $(2,2)$-tight graph with at least one edge. Then either
\begin{enumerate}
\item[(i)] there exists a proper $(2,2)$-tight subgraph $H \subset G$ such that no vertex $v \in V(G \setminus H)$ is adjacent to more
than one vertex in $H$, or

\item[(ii)] $G$ is a Laman-plus-one graph.
\end{enumerate}
\end{lem}
\begin{proof}
Suppose that $G$ is not Laman-plus-one. Then there is a proper subgraph $J \subset G$ such that $f(J)=2$ and we may choose $J$ minimal
(with respect to this property) with $|E(J)| \geq 1$. Since $G$ is not Laman-plus-one, for any edge $e \in E(G)$ there is a subgraph
$H \subseteq G \setminus e$ such that $f(H)=2$. In particular we may choose $e \in E(J)$ and we may choose $H$ maximal in $G \setminus e$
such that $f(H)=2$. We have $|V(H)| < |V(G)|$
because otherwise $f(G)=f(H)-1=1$.

Suppose $H$ does not satisfy property (i). Then there are vertices $a,b \in V(H)$ and $v \in V(G \setminus H)$ such that edges
$av,bv \in E(G)$. If $av,bv \neq e$ then $f(H \cup av,bv)=2$ and $H \cup av,bv \subset G \setminus e$ which contradicts the maximality of $H$.
We may assume therefore that $av=e$. This implies $a \in V(H \cap J)$ because $e \in E(J)$. By Lemma \ref{l:union}
(with $r=2, G_1=H$ and $G_2=J$) we have $f(H \cup J)=f(H \cap J)=2$. 
The minimality of $J$ implies $V(H \cap J)= a$ and then $f(H \cup J \cup bv)=1$ which contradicts the $(2,2)$-sparsity of $G$.
\end{proof}

Note that, as with the $K_4$ examples above, two $(2,2)$-tight graphs
may be joined at a common  vertex, or may be joined
by two disjoint edges to create a new $(2,2)$-tight graph. Thus the class of $(2,2)$-tight graphs is closed under these two joining operations.
Using these two moves with $K_4$ one obtains large graphs which are
$(2,2)$-tight which have no inverse Henneberg move
to a $(2,2)$-tight graph. However these join moves, together with the Henneberg moves, are not sufficient to generate all $(2,2)$-tight graphs. The graph in Figure \ref{3conmi2}, or indeed any $3$-connected $(2,2)$-tight graphs with no inverse Henneberg move, can not be reduced using the inverse of either of these two joining operations.

\begin{center}
\begin{figure}[ht]
\centering
\includegraphics[width=4cm]{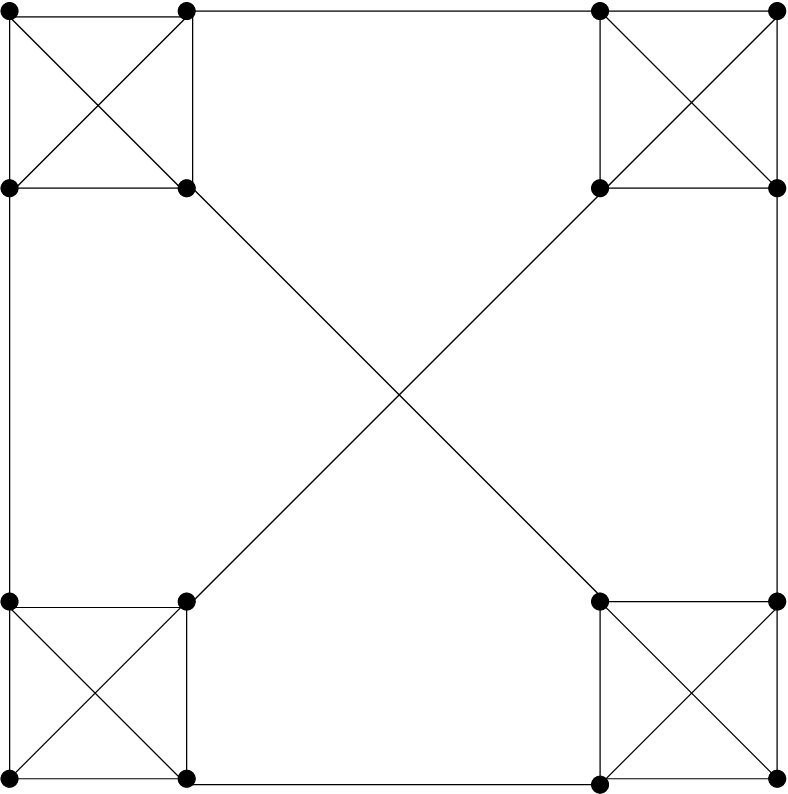}
\caption{A $3$-connected $(2,2)$-tight graph with no inverse Henneberg move.}
\label{3conmi2}
\end{figure}
\end{center}

The following contraction move, which is a companion to the last lemma,
will be used in the proof of Theorem \ref{t:type2graph}.

Let $G$ be $(2,2)$-sparse and let $H$ be a proper subgraph with $f(H)=2$.
Write ~$G/H$~ for the multigraph in which $H$ is contracted to  a single vertex. This is  the graph for which the vertex set is $(V(G)\setminus V(H))\cup \{v_*\}$ and the edge set is $(E(G)\setminus E(H))\cup E_*$
where $E_*$ consists of the edges $(v,v_*)$ associated with edges $(v,w)$
with $v$ outside $H$ and $w$ in $H$.
If $G$ is $(2,2)$-sparse then so is $G/H$ if it happens to be a simple graph.

To see that $G/H$ is $(2,2)$-sparse let $K \subseteq G/H$ and let $\hat{K} \subset G$
be the subgraph for which
\[
V(\hat{K})= (V(K)\setminus \{v_*\})\cup V(H), \quad E(\hat{K})=\pi_e^{-1}(E(K))\cup E(H)
\]
where $\pi_e: E(G)\setminus E(H) \to E(G/H)$ is the natural map.
Since $\pi_e^{-1}:E(K) \to E(G)$ is one-to-one it follows that
\[
2 \leq f(\hat{K})=2(|V(K)|-1)+2|V(H)|-(|E(K)|+|E(H)|)
\]
\[
= f(K)-2 +f(H) = f(K)
\]
as desired.

Note that the simplicity of $G/H$ is guaranteed by the simple condition that no vertex $v \in V(G \setminus H)$ is adjacent to more than one vertex in $V(H)$.

We now identify a natural set of moves through which we may derive from $K_4$ all the
$(2,2)$-tight graphs with at least one edge.

{\begin{defn}\label{d:extensionmove}
An extension move $H \to G$ in the class of simple graphs is an inclusion map
$H \to G$ such that  $G/ H$ is a simple graph. A $(2,2)$-tight extension move (or simply an extension move if the context is clear) is an extension move $H \to G$
for which $H, G$ and $G/ H$ are $(2,2)$ tight.
\end{defn}
}

{
\begin{thm}\label{t:type2graph}
Every  $(2,2)$-tight simple graph with more than one vertex can be obtained from $K_4$
through a finite sequence of Henneberg moves and $(2,2)$-tight extension moves.
\end{thm}}

{
\begin{proof} Suppose that there is a nonsingleton graph
$G_*$ which is $(2,2)$-tight
and which is not derivable using the three moves.
Suppose also that $G_*$ has a minimal number of vertices
amongst such graphs.
By Proposition \ref{p:lamanplusonegraph}
$G_*$ is not Laman-plus-one. Thus (ii) in  Lemma \ref{l:alternative} holds for some subgraph $H$.
But in this case the quotient $G/ H$ is simple and by previous remarks it is $(2,2)$
tight. Thus $G_*$ has an inverse extension move, contrary to its definition.
\end{proof}}

We now use the inductive characterisations to obtain a straightforward deduction of the spanning tree characterisations mentioned in the introduction, namely
the equivalences between (i) and (ii) in the next two theorems. The first characterisation is due to Recski \cite{Rec} where the proof is based on determinental expansions. See also Whiteley \cite{whi-1} for a matroidal proof. The second characterisation was obtained by Nash-Williams in \cite{N-W}, through analysis of critical set partitions, where additionally
$k$-fold spanning tree decompositions are also characterised.

A graph $H=(V,E)$ is said to be an edge-disjoint union of $k$ spanning trees if there is a partition $E_1, \dots , E_r$ of $E$ such that the subgraphs $(V, E_1), \dots , (V, E_r)$ are (connected) trees.

\begin{thm}\label{t:nash-williams1}
The following assertions are equivalent for a (simple) connected graph $G$.
\begin{enumerate}
\item[(i)] $G$ is $(2,3)$-tight.
\item[(ii)] If $G^+$ is the graph (or multi-graph) obtained from $G$ by adding an edge (including doubling  an edge)  then $G^+$  is an edge-disjoint union of two spanning trees.
\item[(iii)] $G$ is derivable from $K_2$ by Henneberg moves.
\end{enumerate}
\end{thm}

\begin{proof}
That (ii) implies (i) is elementary (as given explicitly in the proof below) while (i) implying (iii) follows from Proposition \ref{p:lamangraph}.
We show by elementary induction that (iii) implies (ii).

Let $G\rightarrow G^{'}$ be a Henneberg $1$ move, adding a degree $2$ vertex $v$, and let $(G^{'})^{+}$ be obtained from $G^{'}$ by addition of an edge $e$ (including doubling).

If $e$ is added to $G$ then we may assume $G+e$ is the union of $2$ edge-disjoint spanning trees. To each of the trees we may add one of the new edges.

The other case is when $e=uv$ for some $u \in V(G)$, indicated in Figure \ref{f:nash}.

\begin{center}
\begin{figure}[ht]
\centering
\includegraphics[width=5cm]{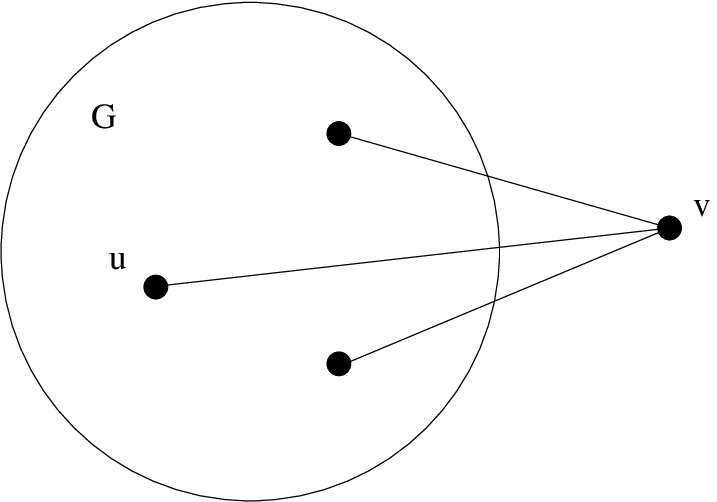}
\caption{$(G')^+$, obtained from $G$ by  Henneberg $1$ move plus added edge $uv$.}
\label{f:nash}
\end{figure}
\end{center}

 Suppose $G^{+}=G\cup f, f=gh,$ decomposes into two edge-disjoint spanning trees $T_{1},T_{2}$.
We now have a decomposition of $G$ into a spanning tree $T_{1}$ and $T_{2}\setminus f$ which is either $(a)$ an edge-disjoint spanning (disconnected) forest or $(b)$ an edge-disjoint (non-spanning) tree.
In case $(a)$ if $v$ is adjacent to vertices in both connected components of $T_{2}\setminus f$ then add both new edges (in the Henneberg move) to $T_{2}\setminus f$ to form $T_{2}^{'}$ and add the ``addition''  edge to $T_{1}$ to get $T_{1}^{'}$.
If $v$ is adjacent to vertices in the same connected component then add one of the new edges (in the Henneberg move) to $T_{1}$ and one to $T_{2}\setminus f$, then add the ``addition''  edge to the other component of $T_{2}\setminus f$ to get $T_{1}^{'}$ and $T_{2}^{'}$.

In case $(b)$ suppose the vertex not in $T_{2}\setminus f$ is $w$. If $w$ is adjacent to $v$ then add $vw$ and some $vx$ to $T_{2}\setminus f$ and add the ``addition''  edge to $T_{1}$ to form $T_{1}^{'}$ and $T_{2}^{'}$.

Finally if $w$ is not adjacent to $v$ then add the new edges one to each of $T_{1}$ and $T_{2}\setminus f$ and add the ``addition''  edge $vw$ to $T_{2}\setminus f$ to get $T_{1}^{'}$ and $T_{2}^{'}$.

By construction in each case $T_{1}^{'}$ and $T_{2}^{'}$ are edge-disjoint spanning trees for $(G^{'})^{+}$ and a very similar elementary argument holds for the Henneberg $2$ move which we leave to the reader.

Since $K_{2}^{+}$ is an edge-disjoint union of two spanning trees the proof is complete.


\end{proof}

\begin{thm}\label{t:nash-williams2}
The following assertions are equivalent for a (simple) connected graph $G$
with at least one edge.
\begin{enumerate}
\item[(i)] $G$ is $(2,2)$-tight,
\item[(ii)] $G$  is an edge-disjoint union of two spanning trees,
\item[(iii)] $G$ is derivable from $K_4$ by Henneberg moves and subgraph extensions.
\end{enumerate}
\end{thm}

See \cite{N&O} for an extension of Theorem \ref{t:nash-williams2}.

\begin{proof}
That (ii) implies (i) is elementary as follows.
Let the two edge-disjoint spanning trees be $T_1=(V,E_1)$ and $T_2=(V,E_2)$. It is a simple property of trees that $|E_i|=|V|-1$ and $|E_i'| \leq |V'|-1$ for all subgraphs $T_i'=(V',E_i')$ of $T$, for $i=1,2$.
Clearly $E$ is the disjoint union of $E_1$ and $E_2$ and  so $|E|=2|V|-2$ and $|E'| \leq 2|V'|-2$ for all subgraphs $G'=(V',E')$ of $G$.

Theorem \ref{t:type2graph} shows that (i) implies (iii) and we now show  that (iii) implies (ii) by induction.

As in the last proof (with simplification due to the absence of edge addition)
the Henneberg $1$ and $2$ moves preserve the spanning trees property of (ii).
Suppose then
that $G/H$ and $H$ decompose into edge-disjoint spanning trees and let $G$ be formed by the graph extension move, where $v_{*} \in G/H$ is replaced by $H$. We  show that $G$ decomposes into edge-disjoint spanning trees.

Note that $v_{*}$ has degree $d \geq 2$. Suppose the two spanning trees for $G/H$ are $T_{1}$ and $T_{2}$ and the two for $H$ are $H_{1}$ and $H_{2}$. Suppose there are $m$ edges incident to $v_{*}$ in $T_{1}$ and $n$ edges incident to $v_{*}$ in $T_{2}$. Call these subsets of edges $E_{1}$ and $E_{2}$ respectively. That is, $E_{i}=\{av_{*}:a \in S_{i}\}$ where $S_i \subseteq V(T_i)$, $i=1,2$.

In the extension move these edges are replaced with edges incident to vertices in $H$. Call these new subsets of edges $E_{1}^{'}$ and $E_{2}^{'}$ respectively,
so that
$$E_{1}^{'}=\{aw_a :a \in S_{1}\},  \quad E_{2}^{'}=\{au_a :a \in S_{2}\}.
$$
Then we claim that $G$ decomposes into two edge-disjoint spanning trees $G_{1}$ and $G_{2}$ where, abusing notation slightly,
$$G_{i}=((T_{i} \setminus E_{i}) \cup H_{i} \cup E_{i}^{'}).$$

It is clear that every edge of $G$ is in $G_{1}$ or $G_{2}$, that no edge is in both, that every vertex is in $G_{1}$ and in $G_{2}$, and  that $G_{1}$ and $G_{2}$ are connected. It remains to show that $G_{1}$ and $G_{2}$ are trees and we need only consider $G_{1}$. Suppose that there is a cycle in $G_{1}$. Then there exists some pair of vertices $a,b \in V(G) \setminus V(H)$ incident to some edges in $E_{1}^{'}$ such that $a$ and $b$ are connected in $G \setminus H$. However this connectedness is necessarily present in $(G/H) \setminus v_{*}$ and so there is a cycle in $G/H$, a contradiction.
\end{proof}

{
\begin{rem}{\rm
In a similar spirit, Crapo \cite{Cra} showed that the class of $(2,3)$-tight graphs are exactly the graphs which have a 3T2
decomposition. This is a decomposition into 3 edge-disjoint trees such that each vertex is in exactly 2 of them and no subgraph with
at least one edge is spanned by subgraphs of two of the three trees. Spanning tree decompositions are of interest because they
produce efficient polynomial time algorithms for checking generic minimal rigidity, whereas algorithms based on checking that all
subgraphs satisfy the independence type are exponential in the number of vertices. See Graver, Servatius and Servatius \cite{GSS} for
more details.  An alternative polynomial time algorithm, applicable to minimally rigid graphs in the plane and to
$(2,2)$-tight graphs, is the elegant pebble game algorithm due to Hendrickson and Jacobs \cite{H&J}, see also \cite{L&S}.
}\end{rem}
}

\begin{rem}
{\rm
The class of $(2,2)$-tight \textit{multigraphs} has been considered by
Ross \cite{Ros} in the setting of periodic frameworks and has been shown to be the relevant
class of graphs  for a Laman type theorem for periodic isostaticity. Here the
flat torus  plays the role of the ambient space
and finite frameworks on it, with possibly wrap-around (locally geodesic edges) model the relevant periodic frameworks.
Interestingly, see \cite{Tay2}, all such graphs derive from the singleton graph by Henneberg $1$
and $2$ moves together with the move of a single-vertex double-edge addition  move (being a variant of the Henneberg $1$ move for multigraphs) and a double-edge variant of the Henneberg $2$ move (arising when, in our earlier notation, $v_k=v_i$ or $v_j$).}
\end{rem}

\section{Frameworks on Surfaces}
\label{sec:fonsurface}
We now consider infinitesimal and continuous rigidity for bar-joint frameworks on general surfaces.  In particular we focus on completely regular frameworks as the appropriate topologically generic
notion, noting that for algebraic surfaces this includes the case of (algebraically) generic frameworks.
It is shown that continuous rigidity and infinitesimal rigidity coincide for completely regular frameworks, a fact which will be a convenience later particularly in the consideration of frameworks on the cylinder.

We remark that the basic theory of the rigidity and flexibility of frameworks on surfaces considered here is a local one in the sense that the concepts and properties depend on the nature of $\M$ near the framework points $p_1, \dots , p_n$.

\subsection{Continuous Rigidity}
\label{contregularsubsec}

Let $\M \subseteq \bR^3$ be a surface. Formally this is a subset with the relative topology which is a two-dimensional
differentiable manifold. However, of particular
interest are the
elementary surfaces which happen to be disjoint unions of algebraic surfaces.

Let $G=(V,E)$ and let $|V|=n$.
A \textit{framework on $\M$} is a framework $(G,p)$
in $\bR^3$
with $G$  a simple connected
graph such that the framework vector $p=(p_1, \dots , p_n)$
has \textit{framework points} $p_i$ in $\M$.
The framework
is \textit{separated} if its framework points are distinct.

The \textit{edge-function} $f_G$ of
a  framework $(G,p)$ on $\M$ is the function
\[
f_G : \M^n \to \bR^{|E|}, \quad f_G(q) = (|q_i-q_j|^2)_{e=v_iv_j}.
\]
This is the usual {edge function} of the free framework in $\bR^3$
restricted to the product manifold $\M^n=\M \times \dots \times \M$
consisting
of all possible framework vectors for $G$. It
depends only on $\M$ and the abstract graph
$G$ and for the moment, without undue confusion, we omit the dependence on
$\M$ in the notation.

In the next definition we write
$(K_n,p)$ for the complete framework on the same set of framework vertices
as $(G,p)$.

\begin{defn}\label{d:ctsrigid}
Let $p=(p_1,\dots ,p_{n})$ and let $(G,p)$ be a framework on the surface $\M$.
\begin{enumerate}
\item[(i)] The \emph{solution space} of  $(G,p)$ is the set  $$ V_\M(G,p) = f_G^{-1}(f_G(p)) \subseteq \M^n$$ consisting of all vectors $q$ that satisfy the distance constraint equations
\[ |q_i-q_j|^2 =  |p_i-p_j|^2, \mbox{  for all edges } {e=v_iv_j}. \]
\item[(ii)] A framework $(G,p)$ on $\M$ is \emph{rigid}, or, more precisely, \emph{continuously rigid},
if for every continuous path $p:[0,1] \to  V_\M(G,p)$ with $p(0)=p$ there exists $\delta >0$ such that $p([0,\delta)) \subseteq V_\M(K_n,p)$.
\end{enumerate}
\end{defn}

It is easy to see that this is equivalent to the following definition, which is simply the standard definition of continuous rigidity
with $\bR^3$ replaced by $\M$.
A framework $(G,p)$ on $\M$ is continuously rigid if it does not have a continuous flex $p(t)$ (a continuous function
$p:[0,1] \to \M^n$ with
$p(0)=p$, $|p_i(t)-p_j(t)|=|p_i-p_j|$ for each edge)
such  that $p(t) $ is not congruent to $p$ for some $t$.

The solution space is topologised naturally with the relative topology
and, as with free frameworks, may be referred to  as
the  realisation space of the constrained framework.

We now take into account the smoothness of $\M$ and the smooth parametrisations of $\M$
near framework points.

Let $h(x,y,z)$ be a rational  polynomial with real algebraic variety $V(h)$ in $\bR^3$.
Assume that $\M$ is a subset of $V(h)$ which is a two-dimensional manifold, not necessarily connected, and let $(G,p)$ be a framework on $\M$ with $n$ vertices as before.
We  associate with the framework the following
{augmented equation system} for the $3n$ coordinate variables of points $q=(q_1, \dots ,q_n)$:
\[
|q_i-q_j|^2 = |p_i-p_j|^2, \quad \mbox{for  } v_iv_j \in E,
\]
\[
h(q_i)=0, \quad \quad \quad \mbox{  for  } v_i \in V.
\]
The solution set for these equations is the solution set $V_\M(G,p)$ which we
also view as the set
\[
\tilde{f}_G^{-1}(\tilde{f}_G(p))
\]
where $\tilde{f}_G $ is the augmented edge function from $\bR^{3n} \to \bR^{|E|+n}$
given by
\[
\tilde{f}_G(q) = (f_G(q), h(q_1), \dots , h(q_n)),
\]
where now $f_G$ is the usual edge function for $G$ defined on all of $\bR^{3n}$,
rather than just on $\M^n$.

More generally let $\M$ be a surface in $\bR^3$ for which there are
smooth functions $h_1, \dots , h_n$ which
determine $\M$ near $p_1, \dots , p_n$, respectively.  Then we define the \textit{augmented
edge function} by
\[
\tilde{f}_G(q) := (f_G(q), h_1(q_1), \dots , h_n(q_n)),\quad  q \in \bR^{3n}.
\]

Suppose for the moment that $(G,p)$ is a free  framework in $\bR^d$.
Write $B(p,\delta)$ for the product $B(p_1,\delta)\times \dots \times B(p_{|V|},\delta)$ of the open balls
$B(p_i,\delta)$ of radius $\delta$ centred at the framework points.
Then $(G,p)$ in $\bR^d$
is \textit{regular} if the point $p$ in the domain of the edge function $f_G:\bR^{dn}\to \bR^{|E|}$ is one
where the derivative function $Df_G(\cdot)$ achieves its maximal rank. This is to say that $p$ is
a \textit{regular point} for this function on $\bR^{3n}$.
The regular points form a dense open set in $\bR^{3n}$, since the nonregular (singular) points are
determined by a finite number of polynomial equations.
By standard multivariable analysis
a regular point $p$ in  $V(G,p)$  has a neighbourhood
$$V(G,p)^\delta = B(p,\delta)\cap V(G,p),$$
which is diffeomorphic to a Euclidean ball in $\bR^k \subseteq \bR^{3n}$ for some $k$.
We take the dimension $k$ as the definition of the (``free'') dimension
$\dim(G,p)$ of the framework.  It follows that
all points $q$ close enough to $p$ are regular
and $\dim(G,q)=\dim(G,p)$.

These facts extend naturally to frameworks on surfaces.

\begin{defn}\label{d:regular} Let  $(G,p)$ be a framework on a smooth surface
$\M$ with local coordinate functions $h_1, \dots , h_n$.

(i) Then $(G,p)$ is \emph{regular} if $p$ is a regular point for the
augmented edge function $\tilde{f}_G$ in the sense that the rank of the derivative
matrix is
{maximal} in a neighbourhood of $p$ in $\bR^{3n}$.

(ii) If $(G,p)$  on $\M$ is regular then its \emph{dimension}
is the dimension of the kernel of the derivative  of the augmented edge function
evaluated at $p$;
$$\dim_\M(G,p)
:=\dim \ker D\tilde{f}_G(p).
$$
\end{defn}

For a simple example of an irregular framework on the sphere one may take a triangular framework whose vertices
lie on a great circle.

The local nature of $\M$ near a regular point $p=(p_1, \dots ,p_n)$ for the complete
graph $K_n$   determines what we
refer to as the number of \textit{ambient degrees of freedom at $p$}.
We define this formally  as $d(\M,p) = \dim_\M(K_n,p)$. Thus
 $d(\M,p) = 3,2,1$ or $0$.

The path-wise definition of continuous rigidity of $(G,p)$ on $\M$ given above is in fact equivalent to the following set-wise formula:
for some $\delta >0$ the inclusion
$$V_\M(K_n,p)^\delta \subseteq V_\M(G,p)^\delta$$
is an equality.
This equivalence for an arbitrary framework $(G,p)$ is a little subtle in that it follows from the local path-wise connectedness of real algebraic varieties. (That is, each point has a neighbourhood which is path-wise connected.)
However for a regular framework $V_\M(G,p)^\delta$
is an elementary manifold, diffeomorphic to a Euclidean ball,
with submanifold $V_\M(K_n,p)^\delta $
and the equivalence is evident.
It follows, somewhat tautologically, that if $(G,p)$ is a regular
framework, then $(G,p)$ is rigid on $\M$ if and only if
$\dim_\M(G,p) = d(\M,p)$.

As in the case of free frameworks the regular framework vectors for a graph form a dense
open set in $\M^n$.
However, the most amenable constrained frameworks are those that
are completely regular in the sense of the next definition.


\begin{defn}\label{d:seqreg}
A framework
 $(G,p)$  on a smooth surface $\M$
is \emph{completely regular}  if $(H,p)$ is regular on $\M$
for each subgraph $H$.
\end{defn}

For an example of a regular framework which is not completely regular
consider the following. Let $\M$ consist of two parallel planes distance $1$
apart
and for the complete graph  $K_6$ let  $(K_6,p)$ be a separated framework with three non-colinear framework points in each plane.
Such continuously rigid frameworks are regular. However if  there are points
$p_i, p_j$ on separate planes at a minimal distance of $1$ apart then  $(K_6,p)$
is not completely regular simply because a triangle subframework with this edge has an extra
independent flex.

One might view the completely regular frameworks as those that are ``topologically generic''
and in examples one can readily identify
a dense open
set of completely regular frameworks.

The next proposition establishes a necessary ``Maxwell count''  condition.
Here $p'$ is the restriction of $p$ to $V(G')$.

\begin{prop}\label{p:minrigidnec}
Let $(G,p)$ be a completely regular minimally rigid framework on a smooth
two-dimensional manifold $\M$. Then  $$ 2n-|E| =\dim(\M,p)$$
and for each subgraph $G'$ with $|E(G')|>0$,
$$
 2|V(G')|-|E(G')| \geq\dim(\M,p').
$$
\end{prop}

\begin{proof}
Let $G_1$ be a spanning tree of $G$ with edges $e_1, \dots ,e_m$ and let $G_k \subseteq G_{k+1}$ be subgraphs of $G$ with $|E(G_{k+1})| = |E(G_{k})|+1$, for
{$1 \leq k \leq r$} where $m+r=|E(G)|$.
Since $(G_1,p)$ is regular we have $\dim_\M(G_1,p)=2n-|E(G_1)|=n+1$. By complete regularity
the dimensions $\dim_{\M}(G_k,p)$ are defined
and for each $k$
\[
\dim_\M(G_k,p) \geq \dim_\M(G_{k+1},p).
\]
Suppose that $(G,p)$ is minimally rigid on $\M$. From continuous
rigidity we have
$\dim_\M(G,p) = d(\M,p)$ and by minimal rigidity
the inequalities are strict.
To see this note that the elementary manifolds
 $V_\M(G_k,p)^\delta$
 are determined by multiple
intersections. For example if $e_{k+1}=(v_i, v_j)$ then, for all small enough $\delta>0$,
\[
V_\M(G_{k+1},p)^\delta=V_\M(G_{k},p)^\delta \cap \{q:|q_i-q_j|=|p_i-p_j|\}.
\]
Thus if there is an equality at the $k^{th}$ step
then removal of $e_{k+1}$ does not affect the subsequent inequalities
and we arrive at the rigidity of $(G\setminus e_{k+1},p)$, contrary to minimal rigidity.

By the strict inequalities and noting that $r=|E|-|E(G_1)|= |E|-(n-1))$ we see that
\[d(\M,p) = \dim_\M(G,p) = n+1 -r =    2n - |E|
\]
as desired.
\end{proof}

\begin{rem}
{\rm
Recall that a generic point $p_1$ for a connected surface $\M$ defined by an irreducible rational polynomial equation
$h(x,y,z)=0$ is one such that every rational polynomial $g$ vanishing at $p_1$ necessarily vanishes on $\M$. One may similarly define a generic framework $(G,p)$ on $\M$ as one for which every rational polynomial $g$ in $3n$ variables which vanishes on the framework vector $(p_1, \dots , p_n)$ necessarily  vanishes on $\M^n$.
Since the set of generic framework vectors is a dense set, generic framework vectors can be found amongst the open set of completely regular framework vectors.}
\end{rem}

\subsection{Infinitesimal Rigidity}
\label{infpaper1sub}
Fix a smooth surface $\M$ in $\bR^3$.

\begin{defn}\label{d:rigiditymatrix}
Let $(G,p)$ be a framework on $\M$ in $\bR^3$ and let
$h_k(x,y,z)=0$ be the local equation  for the surface $\M$
in a neighbourhood of the framework point $p_k$, for $1 \leq k \leq |V|$.
The \emph{rigidity matrix}, or \emph{relative rigidity matrix}, of $(G,p)$ on $\M$
is the $|E|+|V|$ by $3|V| $  matrix
defined in terms of the derivative  of the augmented edge-function
$\tilde{f}_G$ as
\[
R_{\M}(G,p) = \frac{1}{2}(D\tilde{f}_G)(p).
\]
\end{defn}

The factor of $\frac{1}{2}$ is introduced for consistency with existing usage
for the rigidity matrices of free frameworks. For example the usual three-dimensional
rigidity matrix $R_3(G,p)$ for $(G,p)$ viewed as a free framework appears as the submatrix
of $R_{\M}(G,p)$ given by the first $|E|$ rows.
In block operator matrix terms we have
\[
R_{\M}(G,p)  =  \begin{bmatrix}R_3(G,p) \\ \frac{1}{2}Dh(p)\end{bmatrix}
\]
where, with $|V|=n$, the mapping $h:\bR^{3n} \to \bR^{n}$ is
$$h=(h_1(x_1,y_1,z_1), \dots , h_n(x_n,y_n,z_n)).$$
Note that the kernel of the matrix $(Dh)(p)$ is
determined by the remaining $n$ rows and is the subspace
of vectors $u=(u_1, \dots , u_n)$ where $u_k$ is tangent to $\M$ at $p_k$.
Thus the kernel of the relative rigidity matrix  is the subspace of
$\ker R_3(G,p)$ (the space of
free infinitesimal flexes) corresponding to tangency to $\M$.
Vectors in this kernel are referred to  as \textit{infinitesimal
flexes} for $(G,p)$ on $\M$. The subspace of \textit{rigid motion flexes}
is defined to be $\ker R_{\M}(K_n,p)$. When $(K_n,p)$ is regular this space has dimension
 $d(\M,p)$.

\begin{defn}\label{d:infrigid} Let $(G,p)$ be a regular framework with $n$ framework
vertices on the smooth surface  $\M$ and  suppose that $(K_n,p)$ is regular. Then $(G,p)$ is \emph{infinitesimally rigid} if
\[
\dim \ker R_{\M}(G,p) = \dim \ker R_{\M}(K_n,p) = d(\M,p).
\]
\end{defn}

The following theorem is useful when contemplating Henneberg moves on frameworks
and the preservation of rigidity which we turn to in the next section.

\begin{thm}\label{t:rigiditythm}
Let $\M$ be a smooth surface in $\bR^3$.
A regular framework $(G,p)$ on $\M$ is infinitesimally rigid if and only if
it is continuously rigid.
\end{thm}

\begin{proof}
Let  $p:[0,1]\to V_\M(G,p)$, as in Definition \ref{d:ctsrigid}, be a
(one-sided) continuous flex
of $(G,p)$ on $\M$. Since $p$ is a regular point, if $(G,p)$ is not
rigid on $\M$ then the inclusion
\[
V_\M(K_n,p)^\delta \subseteq V_\M(G,p)^\delta
\]
is proper for all small enough $\delta >0$. Since this is an inclusion of elementary
smooth manifolds there exists a differentiable two-sided flex
$p(t), t \in (-1,1)$ taking values in the difference set (for $t\in (0,\delta))$. Moreover $p(t)$ may be chosen so that $p'(0)$ is not in the tangent
space of  $V_\M(K_n,p)^\delta$ at $p$.
Note that the derivative vector $p'(0) = (Dp)(0)$ in $\bR^{3n}$
lies in the kernel of $R_{\M}(G,p)$. Indeed, if $d_k$ denotes the squared length of the
$k^{th}$ edge of $(G,p)$ then
 we have
\[
\tilde{f}_G \circ p(t) = \tilde{f}_G(p(t)) = (d_1, \dots , d_{|E|}, 0, \dots ,0),
\]
a constant function, and so the derivative
(column matrix) $D(\tilde{f}_G \circ p)(0)$ is zero. By the chain rule
and noting that $p(0)=p$
this  is equal to the matrix product
$(D\tilde{f}_G)(p).(Dp)(0)$.
Thus the flex $v = (Dp)(0)$ is an infinitesimal flex of $(G,p)$ on $\M$ which is not
in $\ker R_{\M}(K_n,p)$. Thus infinitesimal rigidity implies continuous rigidity.

On the other hand continuous rigidity implies equality, for sufficiently small $\delta$, for the elementary manifold inclusion above, and hence equality of the tangent spaces at $p$. This equality corresponds to infinitesimal rigidity.
\end{proof}

In the next section we consider minimally continuously rigid completely regular frameworks.
In view of the theorem above these coincide with the class of minimally infinitesimally rigid
completely regular frameworks. As in the case of free frameworks, we also say that
$(G,p)$ is \textit{isostatic} on $\M$ if it is minimally infinitesimally rigid.

\begin{thm} Let $(K_{|V|},p)$ be regular and let   $(G,p)$ be a completely regular framework on the smooth surface  $\M$. Then $(G,p)$ is isostatic if and only

(i) \[
\rank R_{\M}(G,p) = 3|V|-d(\M,p),
\]
and
(ii)
$$ 2|V|-|E| =d(\M,p).$$
\end{thm}

\begin{proof}
From the definition a framework is infinitesimally rigid if and only if
$$\rank R_{\M}(G,p) = 3|V| - \dim \ker R_{\M}(G,p)= 3|V|-d(\M,p).$$
If $(G,p)$ is minimally infinitesimally rigid then by the last theorem and the hypotheses it is also minimally continuously rigid. Thus (ii) holds by Proposition \ref{p:minrigidnec}.
It remains to show that if (i) and (ii) hold then the framework, which is infinitesimally rigid (by (i)) is  minimally infinitesimally rigid. This follows since
if $E'\varsubsetneq E$ and $((V,E'),p)$ is rigid then
$|E'|+|V|$ is greater than or equal to the row rank and so
$|E|+|V| >  3|V|-d(\M,p)$ and  $ 2|V|-|E| < d(\M,p)$.
\end{proof}

{
\begin{rem}
 {\rm
Note that for the circular cylinder $\M$ we have
\[\dim_\M(K_2,q) = \dim_\M(K_3,r)= 3 \mbox{ and }  \dim_\M(K_4,p) =2,
 \]
when these frameworks are completely regular. Although each of these frameworks is continuously rigid according to our definition, the graphs $K_2, K_3$ are too small to reveal the flexibility constraints
that the cylinder imposes on larger frameworks. These become evident only for $K_4$ and it is from $K_4$ that we can build rigid frameworks with Henneberg moves.
Indeed
if $G$ is the double triangle graph obtained from $K_3$ by a Henneberg move, then a generic
framework $(G,p)$ is not continuously or infinitesimally rigid.
In fact $G=K_4\setminus e$ and we see that a full ``rotation'' (flex) of $(G,p)$
on the cylinder passes through noncongruent realisations of the ``unrotatable'' framework  $(K_4,p)$.
}
\end{rem}
}

\begin{rem}\label{plausibleremark}
{\rm
Let $(K_4,p)$ be  a
separated  regular realisation of $K_4$ in $\bR^3$.
Then a specialisation of six vertex coordinates is sufficient
to remove all continuous nonconstant flexes of $(K_4,p)$.
If the framework vertices  are all constrained to a smooth surface $\M$ then a
specialisation of at most three
equations is needed to remove all continuous flexes. That three may be necessary
can be seen when $\M$ is a plane, or a union of parallel planes, or when $\M$
is a sphere, or a union of concentric spheres. Let us
define the  \textit{degrees of freedom}
$d(\M)$ of the surface  $\M$
as the minimum number of vertex coordinate specifications necessary to
remove the rigid motions
of all proper completely regular realisations of $K_4$ on $\M$. Thus, for the sphere and
the plane there are $3$ degrees of freedom,
for the infinite circular cylinder there are $2$, and for
many familiar surfaces with only rotational symmetry, such as
cones and tori, there is one degree of freedom.
The degrees of freedom of $\M$  coincides
with the minimum value of $\dim_\M(K_4,p)$ as $p$ ranges over separated completely regular quadruples in $\M$.
In light of this, and our Laman style theorem for the cylinder,
a plausible conjecture is the following: \textit{for reasonable manifolds the graphs for which {every} completely regular framework on $\M$ is continuously rigid are those that are $(2,d(\M))$-tight, together with a number of small exceptions.}}
\end{rem}

\section{Henneberg Moves on Constrained\\ Frameworks}
\label{sec:henmconf}
We now work towards combinatorial (Laman type) characterisations
of rigid frameworks on some elementary surfaces.
The proofs follow a common scheme in which we are required to

(i) establish an inductive scheme for the generation of the graphs in the appropriate
class $\C$ for the surface, where the scheme employs moves of Henneberg type or other moves such as graph extensions,

(ii) show that the moves for $\C$ have their counterparts for frameworks on $\M$
in which minimal rigidity is preserved.

\begin{rem}\label{matroidremark}\rm{
We remark that in the case of algebraic manifolds one may define for each graph $G$
the rigidity matroid $\R_{\M}(G,p)$, determined by a generic framework vector,
as the vector matroid induced by the rows of $R_{\M}(G,p)$.
Thus realising the proof scheme amounts to  the determination of a matroid isomorphism
between $\R_{\M}(G,p)$ and the matroid defined by maximal independence counting
in $G$. Further in the case of the plane, combining this with Laman's theorem shows that the vector matroids $\R_{\M}(G,p)$ and $\R_2(G,p')$ (the standard $2$-dimensional rigidity matroid) are isomorphic.  See  \cite{GSS}, \cite{Jac}. This is perhaps surprising since these matroids are induced by matrices of different sizes. However the isomorphism can be seen by considering the $|V|$ rows in $R_{\M}(G,p)$ as fixed (independent) and identifying the $|E|$ rows in $R_{\M}(G,p)$ with the $|E|$ rows in $R_2(G,p')$. Of course it is only in the case of planes and spheres that such an identification can be made.}
\end{rem}

Let $G \to G'$ be the Henneberg 2 move at the graph level
in which the edge $e = v_1v_2$ is broken at a new vertex $v_{n+1}$
and in which the new edge $v_3v_{n+1}$ is added.
Let   $p=(p_1,\dots ,p_n)$.
A  \textit{Henneberg 2 framework move} $(G,p) \to (G',q)$, with  $(G',q)$ also on $\M$,
is one for which the  edges that are common to both $G$ and $G'$ have the same length.

In  constructions of such moves the framework points
$q_1, \dots , q_n$ may usually be taken close to $p_1, \dots , p_n$. Indeed a Henneberg $2$ framework move will arise from a sequence
\[
(G,p) \to (G\setminus e, p) \to (G\setminus e, p(t)) \to (G',(p(t),p_{n+1}(t)))=(G',q)
\]
where the middle step takes place by a small flex on $\M$ and the final step is determined
by a location of $q_{n+1}=p_{n+1}(t)$ on $\M$, with $t$ small, for the vertex $v_{n+1}$.
We  consider
$(G,p)$ also to be minimally rigid
so that $(G\setminus e, p)$ has one degree of freedom, in the sense that
the local solution space $V(G,p)^\delta$ is
a manifold of dimension $d(\M,p)+1$.

To clarify the consideration of such Henneberg framework moves which preserve minimal (continuous) rigidity we
first consider
frameworks in the plane under the requirement of a simple geometric noncolinearity condition.

\begin{prop}\label{p:hennplane}
Let $\delta >0$, let $(G,p)$ with  $p=(p_1, \dots ,p_n), n\geq 2$ be a completely regular minimally  rigid framework in the plane
with no three of the framework points colinear, and let
$G \to G'$ be a Henneberg 2 move. Then there is a  completely regular minimally rigid framework
$(G', p')$, with $p'=(p_1', \dots ,p_n', p_{n+1}')$, with no three of the framework points colinear, and $|p_1-p_i'|<\delta $, for $1 \leq i \leq n$.
\end{prop}

\begin{proof} Consider the depleted
framework $(G\setminus e, p)$ with $e=v_1v_2$.
By minimal rigidity
and complete regularity this framework has one degree of freedom modulo
ambient isometries or, more precisely,
$\dim(G\setminus e, p)=4$.
Consider the $1$-dimensional subset $\N$ of $V(G\setminus e, p)^\delta$ consisting of points
$q$ for which $q_1=p_1$ and $q_2$ lies on the line though $p_1$ and $p_2$.
Thus there is a continuous flex $p(t) = (p_1(t), \dots , p_n(t))$ in $\N$ for which
$|p_1(t) - p_2(t)|$ is  decreasing on some small interval
$[0, \delta)$
and we may also assume that this flex  is differentiable.
Now note that this ``normalised'' flex $p(t)$ extends to a flex of the enlarged framework $((G\setminus e)^+, p^+)$ formed by introducing $p_{n+1}$   on the line segment $[p_1,p_2]$, with the two new edges, $[p_1, p_{n+1}]$ and $[p_2, p_{n+1}]$. See Figure \ref{graph1star} and Figure \ref{graph2work}.

\begin{center}
\begin{figure}[ht]
\centering
\includegraphics[width=3cm]{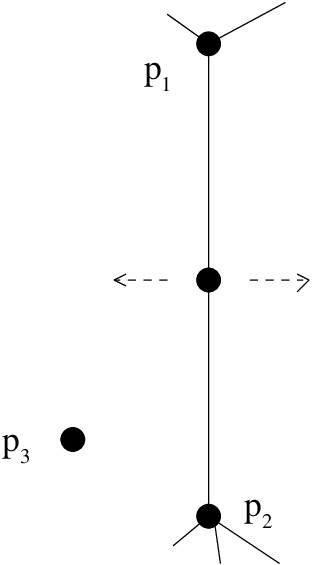}
\caption{Splitting the edge $[p_1, p_2]$.}
\label{graph1star}
\end{figure}
\end{center}

\begin{center}
\begin{figure}[ht]
\centering
\includegraphics[width=3.5cm]{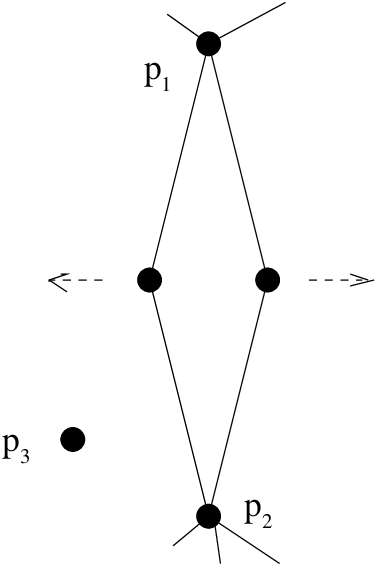}
\caption{The two flexes of $((G\setminus e)^+, p^+)$.}
\label{graph2work}
\end{figure}
\end{center}

There are precisely two such extensions, according to the sense of motion of the hinge point
$p_{n+1}$. It follows from the noncolinearity of $p_1, p_2 $ and $p_3$
that for at least one of these flexes
the separation $s(t)=|p_3(t) - p_{n+1}(t)|$ is a non-constant function on every
interval $[0,\delta]$ for all $\delta < \delta_1$, for some $\delta_1$.
(See Lemma \ref{l:hinge} for a formal proof.)
Since the flex is differentiable $s(t)$ is strictly decreasing or increasing
on a small interval $(0, \delta_2)$. Choose $t$ in this interval and   add the edge
$[p_3(t), p_{n+1}(t)]$  to create the
framework $(G',(p(t),p_{n+1}(t)))$. By construction this is continuously rigid
since there is no nonconstant normalised flex (with $p_1$ fixed and $p_{n+1}$ moving on the
line through $p_1$ and $p_2$).
It also follows readily from the openness of the set of completely regular framework vectors that, for sufficiently small $t$, $(G',(p(t),p_{n+1}(t)))$ is completely regular.
\end{proof}

In the ensuing discussion we focus on continuous flexes and the intuitive device of
hinge separation which we expect to be useful for general manifolds.
However, there are alternative approaches for algebraic surfaces  based on flex specialisation at generic points. (See \cite{NOP2} for example.) We illustrate this with the following alternative
proof to the generic framework variant of the proposition above. Note that
Proposition
\ref{p:hen2isostatic} together with Proposition \ref{p:lamangraph} provide a short proof of the interesting (sufficiency) direction of Laman's theorem.

\begin{prop}\label{p:hen2isostatic}
Let $G \to G'$ be a Henneberg $2$ move and let $(G,p)$ and $(G', p')$
be generic frameworks on the plane with $G$ a Laman graph. If $(G,p)$ is isostatic
on the plane
(minimally infinitesimally rigid) then so too is $(G', p')$.
\end{prop}

\begin{proof} As before we let $v_1, v_2, v_3$ and $v_h$ be the vertices involved in the Henneberg move for the edge $v_1v_2$.
Suppose that $(G', p')$
is not isostatic. Since $G'$ is a Laman graph it follows that the rank of the rigidity matrix $R_2(G',p')$ is less than $2|V|-3$. Since $p'$ is generic this is the case for any specialisation
of $p'$ and in particular for $p'=(p_1,\dots , p_n, p_h)$ where
$(p_1,\dots , p_n)=p$ and $p_h$ is any point on the open line segment
from $p_1$ to $p_2$;
 $p_h=ap_1+(1-a)p_2$, with $0<a<1$.
Thus there is an infinitesimal flex $u'=(u_1, \dots ,u_n, u_h)$ for $(G',p')$
which is not a rigid motion flex.
We have
\[
\langle u_1-u_h, p_1-p_h \rangle =0,  \quad   \langle u_2-u_h, p_2 -p_h \rangle =0,
\]
and so by the colinearity of $p_1, p_2, p_h$,
\[
\langle u_1-u_h, p_1-p_2 \rangle =0,  \quad   \langle u_2-u_h, p_1 -p_2 \rangle =0.
\]
Thus
\[
\langle u_1-u_2, p_1-p_2 \rangle =0,
\]
and so the restriction to $(G,p)$, namely $u=(u_1, \dots ,u_n)$, is an infinitesimal flex of $(G,p)$.

By the hypotheses $u$ is an infinitesimal rigid motion of $(G,p)$.
In particular the restriction $u_r= (u_1, u_2, u_3)$ of $u$ to the
triangle $p_1, p_2, p_3$ is a rigid motion infinitesimal flex for some isometry
$T:\bR^2 \to \bR^2$. But note that
$u_r$ is also a restriction of $u'$, and  the triangle is noncolinear, so it follows  that
$
u_h$ must be equal to $au_1+(1-a)u_2.
$
Thus $u'$ itself is a rigid motion flex, also associated with $T$, contrary to assumption.
\end{proof}

\subsection{Hinge Frameworks}
In the proof of Proposition \ref{p:hennplane} the key point is that
the edge $[p_1,p_2]$ is replaced by two edges $[p_1,p_{n+1}]$ and $[p_{n+1},p_2]$
which can ``hinge'' in \textit{two} directions when $p_1, p_2$ flex towards each other.
Similarly, for frameworks on surfaces we examine the placement of $p_{n+1}$ at such special points. With two flex directions (and with a version of the non-colinearity
condition for $p_3$ relative to $p_1$ and $p_2$) we obtain a ``proper
separation'' of $|p_3(t) - p_{n+1}(t)|$  on all small enough intervals for at least one of
these directions.
This last idea is
formalised rigourously, in a three-dimensional setting, in assertion (ii) of the hinge framework lemma below. While it seems evident that, roughly speaking, generically
one can make a rigidifying Henneberg $2$ move, it should be borne in mind that
the motion $p_3(t)$ is undetermined (and can be an arbitrary algebraic curve \cite{Kem}).
Thus one needs some systematic method for avoiding exceptional placements of $p_{n+1}$
in which there is no proper separation.

Let $H$ be the cycle graph with four edges and four vertices $v_1, \dots ,v_4$
in cyclic order. Let $(H,q)$ be a framework in $\bR^3$ with
$q=(a,b,c,d)$ where  $a, \dots ,d$ are points in $\bR^3$ with $|a-b|=|a-d| \neq
 0$ and $ |c-b|=|c-d|\neq 0$. We refer to this as a  \textit{hinge framework} and when $b=d$ as a \textit{closed hinge framework}.

\begin{lem}\label{l:hinge}
Let $q(t)=(a(t), b(t), c(t), d(t))$ be a continuous flex of the
closed hinge framework $(H,q)$ in $\bR^3$, with $q(0)=q$, such that
$$t \to |b(t)-d(t)|$$
is nonconstant on every interval $[0,\delta), \delta < 1$, and
let $v(t)$ be a path in $\bR^3$.
Then one of the following holds:

(i) for some $\delta > 0$ and all $t \in [0,\delta)$
\[
\langle b(t)-d(t),a(t)-v(t) \rangle = 0,
\]

(ii) at least one of the functions $t\to |v(t)-b(t)|$, $t\to |v(t)-d(t)|$
is nonconstant on all intervals $[0,\delta)$ for $\delta$ less than some $\delta_1$.
\end{lem}

\begin{proof}
Suppose that (ii) fails and the functions are constant
in some interval $[0,\delta)$. Since $b(0)=b=d=d(0)$ the functions are
equal in this interval.
Then, on this interval,
$$\langle v(t)-d(t),  v(t)-d(t)\rangle = \langle v(t)-b(t),  v(t)-b(t)\rangle$$
and so $\langle v(t),b(t)-d(t)\rangle = (|b(t)|^2-|d(t)|^2)/2$. The same is true
with $v(t)$ replaced by $a(t)$ and so (i) follows.
\end{proof}

Consider a fixed value of $t>0$ and
note that apart from the exceptional case when $b(t)$ and $d(t)$ coincide
and  $a(t),b(t),c(t)$ are  colinear there is  a unique plane $P(a(t),b(t),c(t))$
which passes through the midpoint of the line segment $[b(t),d(t)]$ and is normal
to the vector $b(t)-d(t)$.
With $r=(x,y,z)$ this is the plane with equation
\[
\langle b(t)-d(t),a(t)-r \rangle = 0.
\]
Because of distance preservation in the flex $q(t)$ note that the plane
$P(a(t),b(t),c(t))$
passes through $a(t)$ and $c(t)$.
For $t=0$ and $a,b, c$ not colinear we define $P(a(0),b(0),c(0))$  simply as
the plane through
$a, b, c$.
In particular, if $a, b $ and $c$ are not colinear and $v(0)$ does not lie on
$P(a,b,c)$ then (i) fails (at $t=0)$ and (ii) holds.

This lemma may be applied, with  the useful conclusion (ii),  whenever one is able to place $p_{n+1}$ on a surface $\M$ in such a way that the added hinge
framework $(H,(p_1, p_{n+1}, p_2, p_{n+1}))$ is ``opened'' (on $\M$) by decreasing separation motion  $p_1(t)$ and $p_2(t)$.

\subsection{Spheres and Planes}

The case of Henneberg $2$ moves on frameworks on concentric spheres and parallel planes
is straightforward in that it follows the format of the proof of Proposition \ref{p:hennplane} for the plane, making use of the hinge lemma at an appropriate point.

\begin{lem}\label{l:spherehingelemma}
Let $\M_1, \M_2, \M_{n+1}$ be concentric spheres.
Let $p_1(t), p_2(t)$ be paths on the spheres
$\M_1$ and $\M_2$ respectively with $p_1=p_1(0)$,  $p_2=p_2(0)$ such that the separation
$|p_1-p_2|$ is not a local maximum or minimum and such that $|p_1(t)- p_2(t)|$ is decreasing.
Let $p_{n+1}\in \M_{n+1}$ be such that
$p_1, p_2, p_{n+1}$ are not colinear and the plane
$P(p_1, p_2, p_{n+1})$ is orthogonal to
the tangent plane to $\M_{n+1}$ at $p_{n+1}$.
Then for some $\delta_1>0$ the closed hinge framework
$H(q_1,q_2,q_3,q_4) = H(p_1,p_{n+1},p_2,p_{n+1})$
has a flex $q(t)$ for $ t \in [0, \delta_1)$ with $q_1(t)=p_1(t)$,  $q_3(t)=p_2(t)$
and
$|q_2(t)-q_4(t)|$ nonconstant on all intervals $[0,\delta), \delta \leq \delta_1$.
\end{lem}

The idea is illustrated in Figure \ref{fig:spheres}.

\begin{proof}Note that as for a single sphere the union $\M$ of the spheres $\M_i$ has
three ambient degrees of freedom. That is $d(\M,p)=3$ whenever $p$ is a separated
framework vector $(p_1, \dots ,p_n)$ with $n>1$.
Without loss of generality the flex may be assumed to be  normalised
so that $p_1(t)$ is fixed on $\M_1$ and $p_2(t)$ moves towards $p_1$ along the shorter arc of a great circle on $\M_2$ (whose plane meets $p_1$).
The hypothesis on $p_{n+1}$ ensures that it lies on a corresponding great circle, and also
that $p_{n+1}$ is not on the radial line through  either of these points, or coincident to them in the case that $\M_1 = \M_{n+1}$ or $\M_2 = \M_{n+1}$.
The conclusion follows readily from the simple geometry of concentric spheres.
\end{proof}

\begin{center}
\begin{figure}[ht]
\centering
\includegraphics[width=6cm]{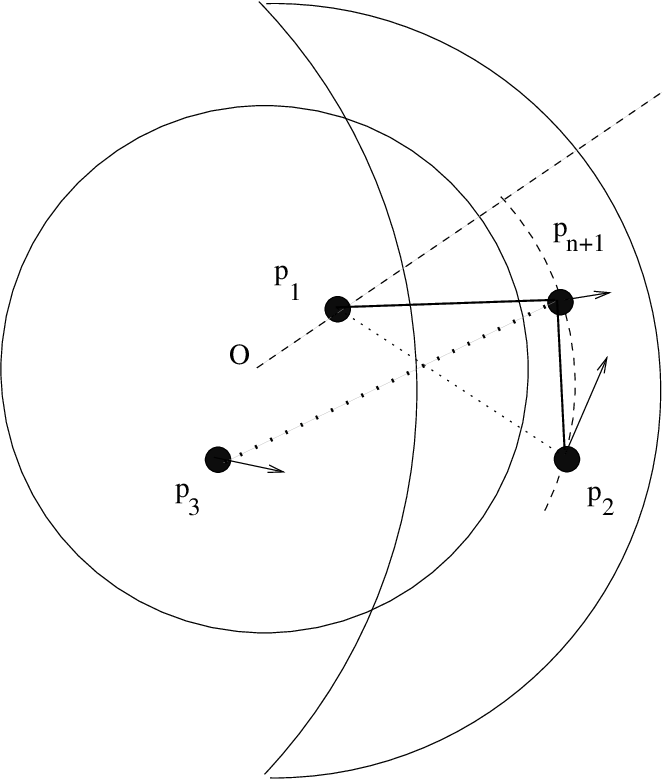}
\caption{With $p_1$ fixed on the inner sphere and $p_3$ moving smoothly
on a concentric sphere, the polar directed smooth motion of $p_2$ on the outer sphere implies infinite initial velocity
at the hinge point $p_{n+1}$ and so strict monotonicity of the separation distance $|p_{n+1}-p_3|$ over a finite time.}
\label{fig:spheres}
\end{figure}
\end{center}

The case of parallel planes has a verbatim statement, with concentric spheres replaced
by parallel planes, and a completely similar proof.

For the Henneberg move construction we require a mild geometric
requirement, being the counterpart  to noncolinearity in the case of a single plane.
More precisely we require that for each  pair $p_i, p_j$ (on separate planes)
the separation $|p_i - p_j|$ is not a local maximum or minimum and that the unique plane $P(p_1, p_2)$ through
the pair, which is orthogonal to the planes (or spheres) of $\M$, meets no  other framework point. We refer to such frameworks as \textit{geometrically generic}.
(In fact one can relax the no extremals condition and treat this class of semigeneric
frameworks separately, although we do not do so here.)

The next Henneberg $2$ framework move lemma has an analogue for the Henneberg $1$ move which is entirely elementary. These framework moves together with standard Laman graph theory are all that are needed for the proof of the sufficiency direction for Theorem \ref{t:lamanplanesandcylinders}.

\begin{lem}\label{l:hennspheres} Let $\M = \M_1 \cup \dots \cup  \M_N$ be a union of parallel planes, or a union of concentric spheres, and let $(G,p)$ be a minimally continuously rigid geometrically generic completely regular framework on $\M$.
Let $\delta >0$, let $s \in \{1, \dots , N\}$ and let
$G \to G'$ be a Henneberg 2 move. Then there is a minimally continuously rigid geometrically generic completely regular  framework
$(G', p')$ on $\M$, with $p'=(p_1', \dots ,p_n', p_{n+1}')$
and $|p_i-p_i'|<\delta $, for $1 \leq i \leq n$, and $p_{n+1}\in \M_s$.
\end{lem}

\begin{proof} The proof has exactly the same form as  that of Proposition \ref{p:hennplane};
with the notation above Lemmas \ref{l:hinge} and  \ref{l:spherehingelemma}
allow for the placement of $p_{n+1}$ so that the flex of $(G\setminus e, p)$
extends to $((G\setminus e)^+,p^+)$ with the separation function $|p_{n+1}(t)-p_3(t)|$ nonconstant on all small
intervals.
\end{proof}

\subsection{Cylinders and Surfaces}
\label{cylinderhen2sub}

We now examine more generally how to place $p_{n+1}$ to create an opening hinge
in the manner of Lemma
\ref{l:spherehingelemma}. This involves the consideration of extremal points
in the sense of the next definition.

\begin{defn}\label{extremaldefn}
Let $\M \subseteq \bR^3$ be a smooth manifold and $p_1, p_2$ distinct points of $\M$.
A point $q\in \M$ is \emph{extremal} for the pair $p_1, p_2$
if there exists a point $w$ on the straight line
through $p_1, p_2$, not equal to $p_1$ or $p_2$, such that $|q-w| < |q'-w|$
for all points $q'\in\M, q \neq q',$ with
$|q-q'| <\delta$, for some $\delta >0$.
\end{defn}
If $q$ is extremal for a pair, as above, then the tangent plane  $T_q$ to $\M$ at $q$ is normal to $q-w$.
Moreover, for small $\delta$ the curve of intersection
\[
\M \cap S(p_1,|p_1-q|)\cap B(q,\delta)
\]
for the surface $S(p_1,|p_1-q|)$ of the closed ball
$B(p_1,|p_1-q|)$ is tangential at $q$ to $S(p_2,|p_2-q|)$ and, apart from the contact point $q$, lies outside the closed ball  $B(p_2,|p_2-q|)$.
Indeed, if this were not the case, for all small $\delta$, then the local closest point property of the extremal point would be violated.



Figure \ref{figgraph3} is indicative of an extremal point $q$, where the plane of the diagram is the plane
$P(p_1, p_2, q)$ through the triple, the curve is in the intersection of this
 plane with $\M$, and the tangent plane $T_q$ to $\M$ at $q$ is orthogonal to
the plane.
Figure \ref{figgraph4} is indicative of the perspective view of such a point and the tangency of  the curves $S(p_1,|q-p_1|)\cap \M$ and  $S(p_2,|q-p_2|)\cap \M$ at $q$.

\begin{center}
\begin{figure}[ht]
\centering
\includegraphics[width=8cm]{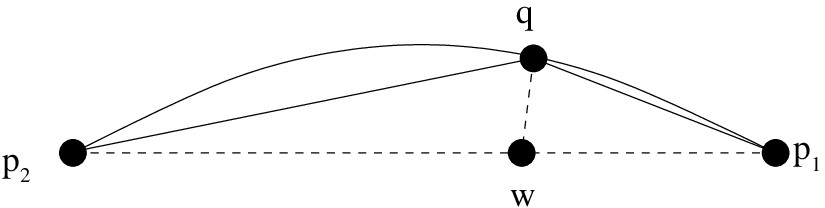}
\caption{An extremal point $q$ for $p_1$, $p_2$, elevation view.}
\label{figgraph3}
\end{figure}
\end{center}

\begin{center}
\begin{figure}[ht]
\centering
\includegraphics[width=6cm]{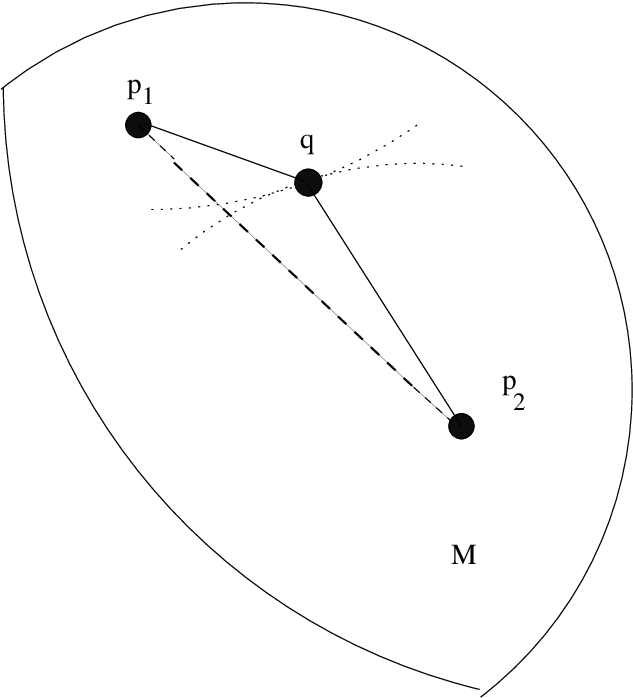}
\caption{An extremal point $q$ for $p_1$, $p_2$ with intersection arcs in $\M$ for
the spheres
$S(p_1,|p_1-q|)$ and
$S(p_2,|p_2-q|)$.}
\label{figgraph4}
\end{figure}
\end{center}

Suppose now that $\M$ is a (circular) cylinder.  If $p_1(t)$ and $p_2(t)$ are continuous paths emanating from $p_1$ and $p_2$ respectively then by rigid motion normalisation we
may assume $p_1(t)=p_1$ for all $t$. Since the cylinder has only two ambient degrees of freedom there is now no further normalisation
available for the adjustment of the motion of $p_2(t)$ or the specification of $p_2'(0)$.
This makes the location of $p_{n+1}$ more problematical in the case that the derivative
of the separation $|p_1 - p_2(t)|$ vanishes at $t=0$. However this complexity
only arises (in our edge-deleted framework context) when $(G,p)$ on $\M$ is infinitesimally
flexible (before the edge deletion). Thus, in view of the equivalence between continuous rigidity and infinitesimal rigidity this difficulty
does not occur for our consideration here. Explicitly, we have the
following condition which expresses  that the separation motion of the pair $p_1(t)$ and  $p_2(t)$ is a
\textit{nontangential separation}:
\[
\langle p_2'(0)-p_1'(0), p_2-p_1 \rangle< 0.
\]

The following simple lemma is needed.  In paraphrase it asserts
the geometrical fact that the tangential departure  $q(t)$ of the point $q$ from
the surface of the ball $B(0,|q|)$, together with an acute-to-$q$ departure $p(s)$ from the origin allows for the solution of the distance equation $|q(t)-p(s(t))|=|q|$ for all $t$ in some small interval, where $t \to s(t)$ is a continuous parameter change.

\begin{lem}\label{l:tangential}
Let $q(s), s\in [0,1]$, be a path in $\bR^3$ starting at $q=q(0)\neq 0$ such that
$\langle q'(0), q\rangle =0 $
 and such that for $s>0$ the path points
$q(s)$ lie outside the closed ball $B(0,|q|)$.
Also, let $p(t), t\in [0,1]$, be a path starting at $p=p(0)= 0$ with
\[
\langle p'(0), q \rangle >0.
\]
Then there is a continuous parameter change $t=t(s)$, for some range $s \in [0,\delta]$, such that in this range
\[
|p(t(s))-q(s)|= |p-q|.
\]
\end{lem}

\begin{proof}Let $f(s,t) = |p(t)-q(s)|^2-|q|^2$.
Consider first the function
\[
t \to f(t,t) = |p(t)-q(t)|^2-|q|^2,
\]
which is zero at $t=0$. In view of the hypotheses, for some positive
number $c$ we have
$\langle p(t), q(t) \rangle \geq ct$ in some small interval $[0,\delta_1]$.
It follows readily that the function $t\to f(t,t)$
is strictly decreasing, and in particular $f(t,t)<0$,
for all $t$ in some small interval $[0,\delta_2]$.

We now see that for fixed $s$ in $[0,\delta_2]$
the function $t \to f(s,t)$ has a strictly positive value at $t=0$ and is negative at
$t=s$. By the intermediate value theorem there is a first point $t(s)$
with  $f(s,t(s))=0$ and moreover, the function $s\to t(s)$ is continuous.
\end{proof}

\begin{lem}\label{l:cylinderhinge}
Let  $p_1, p_2$ be
distinct points on
a cylinder $\M$
such that the line segment from $p_1$ to $p_2$ does not lie in $\M$.
Let $p_2(t)$ be a path on
$\M$ with $p_2(0)=p_2$ such that
\[
\langle p_2'(0), p_2-p_1 \rangle< 0.
\]
Then there is an extremal point $p_{n+1}$ for the pair $p_1, p_2$
such that the closed hinge framework
$H(q_1,q_2,q_3,q_4) = H(p_1,p_{n+1},p_2,p_{n+1})$
has a flex $q(t), t \in [0, \delta_1)$, on $\M$, with $q_1(t)=p_1$,  $q_3(t)=p_2(t)$
and
$|q_2(t)-q_4(t)|$ a nonconstant function on all intervals $[0,\delta), \delta \leq \delta_1$,
for some $\delta_1 >0$.
\end{lem}

\begin{proof}
In view of the discussion above we may choose an extremal point $q$ for the point pair $p_1, p_2$ such that
\[
\langle p_2'(0), p_2-q \rangle< 0.
\]
For example, $q$ may be chosen close to $p_1$, as a local closest point to a point $w$ close to $p_1$.
Let $q(t), t \in [-1,1]$ be a parametrisation of the curve
\[
\M \cap S(p_1,|p_1-q|)\cap B(q,\delta)
\]
for appropriate $\delta>0$, as above.
Apply Lemma \ref{l:tangential} to the path pair $p_2(t), q(t), t \in [0,1]$,
to create $q_2(t)$. Similarly,  use the path pair $p_2(t),\\ q(-t), t \in [0,1]$
to create  $q_4(t)$ and the proof is complete.
\end{proof}


A similar hinge construction lemma holds
for frameworks supported on a union $\M$ of concentric cylinders $\M_i$.
The only new aspect is that the extremal point
$q$ must be chosen on a preassigned cylinder $\M_k$ of $\M$ and we must maintain the inequality
\[
\langle p_2'(0), p_2 - p_1 \rangle >0.
\]
when $p_1$ is replaced by $q$. Maintaining this inequality corresponds to
choosing $q$ in  the halfspace of points $z$ with
$
\langle p_2'(0), p_2 - z \rangle >0.
$
To see that this is possible note
that the line of points $w_t=p_2 + t(p_1-p_2), t \in \bR$,
is not parallel to the common cylinder axis (by assumption) and also that the line
is not orthogonal to $p_2'(0)$. Thus for all $t$ large $w_t$ lies in the half space.
Since the cylinder $\M_k$ passes through the half space it follows (from simple geometry)
that for large enough $t$ the closest point $q_t$ on $\M_k$ to $w_t$ will also lie in the half space, as required.

\section{Combinatorial Characterisations of Rigid\\ Frameworks}
\label{sec:combcharrf}

%
%
%

We now obtain variants of Laman's theorem for bar-joint frameworks constrained to parallel planes, to concentric spheres
and to concentric cylinders. In each case the proof scheme is the same.

\begin{thm}\label{t:lamanplanesandcylinders}
Let
$\M_1, \M_2, \dots ,\M_N$ be parallel planes or concentric spheres in $\bR^3$
with union $\M$, let $G$ be a simple connected graph and let $\pi :V \to \{1,\dots ,N\}$. Then $G$ admits a minimally rigid completely regular framework $(G,p)$ on $\M$, with $p_k \in \M_{\pi(v_k)}$ for each $k$, if and only if $G$ is a Laman graph.
\end{thm}

\begin{proof} Section \ref{sec:fonsurface} shows that the
Laman counting conditions are necessary. For sufficiency
note that there are minimally rigid completely regular frameworks $(K_2,p)$ with
$p_1, p_2$ placed on any pair $\M_i, \M_j$. It is entirely elementary, as made explicit for the case of a cylinder below, that the Henneberg $1$ move preserves minimal rigidity. Thus the constructions of the last section together with the graph theory of Section \ref{sec:grapht} lead to the stated framework realisations if $G$ is a Laman graph.
\end{proof}

We now turn to the proof of a Laman theorem for the cylinder for which we require the following matricial companion to the rigid graph extension move and an elementary argument for the Henneberg $1$ move.

\begin{lem} \label{l:rigidmatrixextension}
Let $\M$ be a union of concentric cylinders in $\bR^3$. Let $H$ be a subgraph of the simple connected graph $G$ such that $K=G/H$ is simple and suppose that $G, H, K$ are $(2,2)$-tight. Suppose that for $H$ and $K$  all completely regular framework realisations on $\M$ are isostatic. Then the same is true of  $G$.
\end{lem}

\begin{proof}
Let $(G,p)$ be completely regular for $\M$ and let $n$ = $|V(G)|$. Let $v_*$ be a fixed vertex of $H$.
Consider the rigidity matrix $R_{\M}(G,p)$ with column triples in the order
of  $v_1, \dots ,v_{r-1}, v_*, v_{r+1}, \dots , v_n$
where $v_1, \dots ,v_{r-1},$   $v_r=v_*$ are the vertices of $H$.
Order the rows of $R_{\M}(G,p)$ in the order of the edges $e_1, \dots ,e_{|E(H)|}$
for $H$
followed by the $n$ rows of the block diagonal matrix whose
diagonal entries are the vectors
$h_1(p_1), \dots , h_n(p_n)$ in $\bR^3$,
followed by the remaining rows for the edges of $E(G)\setminus E(H)$.
Note that the submatrix formed by the first  $|E(H)|+r$ rows  is the
$1$ by $2$ block matrix $[R_\M(H,p)\,  0]$.

Suppose, by way of contradiction that $G$
is not isostatic. Since $2n-|E|=2$ there is a vector $u$ in the kernel
of  $R_{\M}(G,p)$ which is not a rigid motion (infinitesimal) flex.
By adding to $u$ some rigid motion flex we may assume that $u_r=0$.
Write $u=(u_H, u_{G\setminus H})$ where $u_H = (u_1, \dots ,u_r)$.
The matrix $R_{\M}(G,p)$ has the block form
\[
R_{\M}(G,p)=\begin{bmatrix}
R_\M(H,p)&0\\X_1& X_2
\end{bmatrix}
\]
where $X =[X_1\, X_2]$ is the matrix formed by the last $|E(G)|- |E(H)|+n-r$
rows. Since $(H,p)$ is isostatic on $\M$ and $R_\M(H,p)u_H=0$ it follows that
$u_H$ is a rigid motion infinitesimal flex. But the coordinate $u_r
=0$ and so  $u_H =0$.

Consider now the framework vector
\[ p'=(p_r,\dots,p_r,p_{r+1},\dots,p_n)\]
in which the first $r$ framework vertices are specialised to $p_r$ and let
\[ p_*=(p_r,p_{r+1},\dots,p_n) \]
be the reduced length framework vector with associated generic framework $(G/H,p_*)$. By the hypotheses this framework is infinitesimally rigid.

The matrix $X_2=X_2(p)$ is square with nonzero vector $u_{G\setminus H}$ in the kernel and so the determinant as a polynomial in the coordinates of the $p_i$ vanishes identically. It follows that $\det X_2(p')$ vanishes identically and that there is a nonzero vector, $v_{G\setminus H}$ say, in the kernel. But now we obtain the contradiction
\[ R_\M(G/H,p_*)((0,0,0),v_{G\setminus H}) =  ((0,0,0),X_2(p'))((0,0,0),v_{G\setminus H})=0.\]
%
%
\end{proof}

An alternative proof of Lemma
\ref{l:rigidmatrixextension}
can be given which is based on a continuity argument and the fact that the infinitesimal flexibility of a single generic framework $(G,q)$ on $\M$ ensures the infinitesimal flexibility of all generic frameworks $(G,q')$ on $\M$. Consider a sequence $p_N$ of generic framework vectors converging to the specialised framework vector $p'$. Arguing by contradiction one obtains a sequence of unit norm flexes $(0,u_N)$ for the generic frameworks $(G,p_N)$ which, by the compactness of the set of unit norm displacements, provides a unit norm flex $u_*=(0,u)$ for the degenerate framework $(G,p')$. This infinitesimal flex gives an infinitesimal flex of $(G/H,p_*)$, contrary to the hypotheses.

For notational convenience in the following lemma we take $\M$ to be the cylinder defined by $x^2+y^2=1$ in $\bR^3$.

\begin{lem}\label{cylinderhen1}
Let $G$ be $(2,2)$-tight and let $(G,p)$ be a minimally rigid regular framework on the cylinder $\M$ with $p=(p_1,\dots,p_n)$ and vertices $v_1,\dots,v_n$. Let $G \rightarrow G'$ be a Henneberg $1$ move which adds the vertex $v_{n+1}$ and edges $v_1v_{n+1},v_2v_{n+1}$. Then there is a regular minimally rigid realisation $(G',p')$ on $\M$ where $p'=(p,p_{n+1})$.
\end{lem}

\begin{proof}
Let $p_i=(x_i,y_i,z_i)$ and for $a=x,y,z$ let $a_{i,j}$ denote the difference $a_i-a_j$. $R_{\M}(G',p')$ has the following form (where the unfilled block matrix in the top left corner is the rigidity matrix $R_\M(G,p)$):

\begin{equation*}
\setcounter{MaxMatrixCols}{20}
\begin{bmatrix}
&&&&&&& 0 & 0 & 0\\
&&&&&&& \vdots & \vdots & \vdots\\
&&&&&&& 0 & 0 & 0\\
x_{1,n+1} & y_{1,n+1} & z_{1,n+1} & 0 & 0 & 0 & \dots & x_{n+1,1} & y_{n+1,1} & z_{n+1,1} \\
0 & 0 & 0 & x_{2,n+1} & y_{2,n+1} & z_{2,n+1} & \dots & x_{n+1,2} & y_{n+1,2} & z_{n+1,2} \\
0 & 0 & 0 & 0 & 0 & 0 & \dots & 1 & 1 & 0
\end{bmatrix}
\end{equation*}

By the structure of $R_{\M}(G',p')$, the minimal rigidity of $(G,p)$ and the regularity of $p'$ the proof is completed by noting that the $3$ by $3$ matrix in the bottom right hand corner has rank $3$.
\end{proof}

\begin{thm}\label{t:lamancylinder}
Let $\M$ be a circular cylinder in $\bR^3$ and let $G$ be a simple connected graph. Then $G$ admits a minimally rigid completely regular framework $(G,p)$ on $\M$ if and only if $G$ is $(2,2)$-tight.
\end{thm}

\begin{proof} Note that the necessity of the condition on the graph follows from Proposition \ref{p:minrigidnec}.

For the sufficiency, first observe that the singleton graph and $K_4$ both have minimally rigid completeley regular realisations on $\M$. 
Theorem \ref{t:type2graph} implies that to complete the proof we need only show that the Henneberg 1, Henneberg 2 and extension operations preserve
minimal rigidity. This is the content of Lemma \ref{cylinderhen1}, the results of the last section and Lemma \ref{l:rigidmatrixextension} 
respectively.

\end{proof}

From the discussion in the last section we also obtain a similar combinatorial characterisation for frameworks on concentric cylinders, with statement and proof in the style of Theorem \ref{t:lamanplanesandcylinders}.
The final ingredient of the proof is to show that Lemma \ref{l:rigidmatrixextension} also holds for the reducible manifold formed by a finite number of concentric cylinders. The proof is as before with the following appropriate form of generic point for $\M$.

Recall first that for the irreducible case, with $p=(p_1,\dots,p_n) \in \M^n \subset \bR^{3n}$, the genericness of $p$ for $\M$ amounts to consideration of the quotient ring $\bQ[x_1,y_1,z_1,\dots,x_n,y_n,z_n]/I$ where $I$ is the maximal ideal generated by the polynomials $h(x_i,y_i,z_i)=0$ defining $\M$. Thus $p$ is generic if the associated field of fractions is isomorphic to $\bQ(p)$. Equivalently, $p$ (and thus $(G,p)$) is generic if the transcendence degree of the field extension $\bQ(p):\bQ$ is $2n$.

The reducible case is similar. Take a surface $\M=\M_1 \cup \M_2$ defined by a product of irreducible rational polynomials $h_i$ so the varieties $V(h_1) \cong \M_1$ and $V(h_2) \cong \M_2$ are irreducible. Let $p=(p_1,p_2)$ for $p_1=(p_{1,1},\dots,p_{1,n}) \in \M_1^n \subset \bR^{3(m+n)}$ and $p_2=(p_{2,1},\dots,p_{2,m}) \in \M_2^m \subset \bR^{3(m+n)}$ (for $n,m \neq 0$) and let $p_{i,j}=(x_{i,j},y_{i,j},z_{i,j})$. For the corresponding indeterminates we have the tensor product decomposition of the quotient ring
\[ \frac{\bQ[x_{1,1},y_{1,1},z_{1,1},\dots, x_{1,n},y_{1,n},z_{1,n},x_{2,1},y_{2,1},z_{2,1},\dots, x_{2,m},y_{2,m},z_{2,m}]}{\langle h_{1,1},\dots,h_{1,n},h_{2,1},\dots,h_{2,m} \rangle}\cong\]
\[ \frac{\bQ[x_{1,1},y_{1,1},z_{1,1},\dots, x_{1,n},y_{1,n},z_{1,n}]}{\langle h_{1,1},\dots,h_{1,n} \rangle} \otimes \frac{\bQ[x_{2,1},y_{2,1},z_{2,1},\dots, x_{2,n},y_{2,n},z_{2,m}]}{\langle h_{2,1},\dots,h_{2,m} \rangle} \]
with each factor ideal prime and hence these integral domains have fields of fractions $\bF_1,\bF_2$ providing the field of fractions $\bF_1 \otimes \bF_2$. The $(n+m)$-tuple $p$ is said to be generic on $\M$ if $\bQ(p) \cong \bF_1 \otimes \bF_2 \cong \bQ(p_1) \otimes \bQ(p_2)$ or equivalently if $\bQ(p)$ has transcendence degree $2(n+m)$ over $\bQ$. Generalising this definition to reducible surfaces with $k$ irreducible components is purely notational.

We note that this definition of genericness is necessarily stronger than taking $p_i$ to be generic on $\M_i$. There are frameworks on reducible surfaces in which each one-manifold sub-framework is generic yet the complete framework is not even regular. For an example, take concentric cylinders and a framework with radial points.

The theorems, of course, extend to the case of rigid frameworks for graphs which contain spanning subgraphs that are $(2,2)$-tight.

\begin{rem}
{\rm
We note that
Whiteley \cite{Whi5} discusses  analogous results for frameworks on the flat (geodesic) cylinder and other flat spaces.
The cylinder context considered  concerns  infinitesimal rigidity
on the cyclic plane and the infinitesimal motion equations derive from equations in the plane. While this keeps some aspects of the rigidity matrix analogous to the plane there is the added feature of geodesic edges which wrap around the cylinder
(for which $k$-frame matroids are introduced to play a role).
}
\end{rem}

\subsection{Cone Graphs}
We say that a graph  $G=(V,E)$ is a \textit{cone graph}
if there is at least one  distinguished vertex $v$ which is adjacent to every other vertex.
The following corollary and its plane variant indicated below is due to Whiteley \cite{Whi2}.
It is well-known that the equivalence of (i) and (ii) does not hold in general
(as the so called double banana graph reveals).

\begin{cor}
Let $G=(V,E)$ be a cone graph. Then the following statements are equivalent.
\begin{enumerate}
\item[(i)] $G$ is $(3,6)$-tight, that is $3|V|-|E| = 6$ and $3|V'|-|E'|  \geq 6$ for every subgraph $G'=(V',E')$ with $|V'|>2$.
\item[(ii)] There is a minimally rigid completely regular framework realisation $(G,p)$ in $\bR^3$.
\end{enumerate}
\end{cor}

\begin{proof}
One can readily see that, with cone vertex $v_1=v$ and $G_0=G\setminus v$,
the set of points $q$ in
$V_{\bR^3}(G,p) \subseteq \bR^{3n}$ with fixed ``centre''  $q_1=p_1$  is in bijective isometric correspondence with the variety $V_\M(G_0, (p_2,\dots ,p_n))$ where $\M$ is
the union of the $p_1$-centred spheres $S(p_1,|p_k-p_1|)$, for $k=2,\dots ,n$. The stated equivalence follows from this correspondence.
\end{proof}

There is a companion result for free bar-joint frameworks in $\bR^3$ subject to the family of constraints that all points are a specified distance from a single plane.  This follows from the parallel planes Laman theorem above. With the plane playing the role of a vertex, for the purposes of counting, the counting requirement is as above.

In a similar way we obtain from the concentric cylinders theorem the following corollary.

Let $G=(V,E)$ be a cone graph with $|V|=n+1$ and with
distinguished cone vertex $v_{n+1}$.
Let $p=(p_1, \dots , p_n)$ be a framework vector, as usual
and let  $p_*$
be a straight line.
Then the triple $(G, p, p_*)$
is a point-line-distance framework  for the cone graph $G$.

In general a point-line graph is a graph $G=(V,E)$ in which we distinguish a bipartition of the vertices $V=V_p \cup V_l$ and a bipartition of the edges $E=E_{pp} \cup E_{pl}$. The notation arises from the point line frameworks context where $v \in V_p$ becomes a point and $v \in V_l$ a line, $e \in E_{pp}$ represents a standard edge (or bar) and $e \in E_{pl}$ represents an edge joining a point to a line. Note that we allow no edges joining two lines and that in the following corollary $|V_l|=1$.

A line has $4$ degrees of freedom and so the natural class
of simple graphs for ``typical'' general point-line distance frameworks are those for which
\[
3|V_p|+4|V_l|-|E| =6
\]
and
\[ 3|V_p(X)|+4|V_l(X)|-|E(X)| \geq 6\]
for any subgraph $X \subset G$ with at least one edge.
We refer to such graphs
as \emph{maximally independent point-line graphs}.
One can define the infinitesimal rigidity of general point-line-distance
frameworks and also generic frameworks in a natural way.

\begin{cor}
Let $G=(V,E)$ be a cone graph, viewed as a point-line graph
$G=(V_p\cup V_l, E) $ with
a single line corresponding to the cone vertex.
Suppose also that the subgraph induced by $V_p$ is connected with at least $4$ points.
Then the following statements are equivalent:
\begin{enumerate}
\item[(i)] $G$ is a maximally independent point-line graph.
\item[(ii)] There is a minimally infinitesimally rigid point-line framework realisation $(G,p)$ in $\bR^3$.
\end{enumerate}
\end{cor}

\begin{proof}
Infinitesimal rigidity of the point-line distance framework is equivalent to the infinitesimal rigidity of the subframework of points constrained to the concentric cylinders (one for each point) determined by
the point line distances. This, by the concentric cylinders theorem above, is equivalent to
\[ 2|V_p|-|E_{pp}|=2 \]
with a similar inequality for subgraphs. For each point there is a point-line edge, and so  $|V_p|=|E_{pl}|$. Thus,
\[ 3|V_p|+4|V_l|-|E| =(2|V_p|+|V_p|)+4-|E_{pl}|-|E_{pp}| =
2|V_p|+4 -|E_{pp}| =6, \]
with an associated inequality for subgraphs, as desired.
\end{proof}

\bibliographystyle{abbrv}
\def\lfhook#1{\setbox0=\hbox{#1}{\ooalign{\hidewidth
  \lower1.5ex\hbox{'}\hidewidth\crcr\unhbox0}}}

\end{document}